\documentclass[11pt,a4paper]{amsart}
\usepackage{amsfonts}
\usepackage{amsmath}
\usepackage{amssymb}
\usepackage{array}
\usepackage{amsthm}
\usepackage{subfigure}
\usepackage{mathdots}
\usepackage{algorithm}
\usepackage[table,x11names]{xcolor}
\usepackage[noend]{algpseudocode}
\usepackage{graphicx}
\usepackage{color}
\usepackage{pgf}
\usepackage{scrpage2}
\usepackage{multirow}
\usepackage{listings} 
\usepackage{tikz}
\usepackage{graphicx}
 \usepackage[foot]{amsaddr}
 		
\newcommand{\N}{\ensuremath{\mathbb{N}}}

\newcommand{\R}{\ensuremath{\mathbb{R}}}

\newcommand{\EE}{\mathbb{E}}

\newcommand{\norm}[1]{\left\Vert #1\right\Vert}

\DeclareMathOperator*{\range}{range}

\DeclareMathOperator*{\trace}{trace}

\DeclareMathOperator*{\Var}{Var}

\DeclareMathOperator{\G}{\mathcal{G}}

\DeclareMathOperator{\Pol}{Pol}

\DeclareMathOperator*{\F}{F}
\DeclareMathOperator{\MSE}{MSE}
\DeclareMathOperator*{\PCA}{PCA}

\DeclareMathOperator{\tvar}{tvar}

\DeclareMathOperator{\Cov}{Cov}

\DeclareMathOperator{\Corr}{Corr}

\newtheorem{theorem}{Theorem}[section]
\newtheorem{lemma}[theorem]{Lemma}
\newtheorem{remark}[theorem]{Remark}
\newtheorem{definition}[theorem]{Definition}
\newtheorem{example}[theorem]{Example}
\newtheorem{corollary}[theorem]{Corollary}
\newtheorem{proposition}[theorem]{Proposition}

\numberwithin{equation}{section}
\numberwithin{table}{section}
\numberwithin{figure}{section}

\newcommand{\bend}{\hspace*{0ex} \hfill \hbox{\vrule height
    1.5ex\vbox{\hrule width 1.4ex \vskip 1.4ex\hrule  width 1.4ex}\vrule
    height 1.5ex}}

\renewcommand{\mathbf}[1]{\ensuremath{\boldsymbol{#1}}}

\newcommand{\rank}{ \operatorname{rank}}

\allowdisplaybreaks

\title[Orthogonal projections for dimension reduction]{On orthogonal projections for dimension reduction and applications in augmented target loss functions for learning problems} 

\date{\today}
\date{}
\author{A.~ Breger$^1$}
\author{J.~I.~Orlando$^2$}
\author{P.~Harar$^3$}
\author{M.~D\"orfler$^1$}
\author{S.~Klimscha$^2$}
\author{C.~Grechenig$^2$}
\author{B.~S.~Gerendas$^{1,2}$}
\author{U.~Schmidt-Erfurth$^2$}
\author{M.~Ehler$^1$}
\address[1]{University of Vienna, Department of Mathematics, Vienna, Austria}
\address[2]{Medical University of Vienna, Department of Ophthalmology, Vienna, Austria}
\address[3]{Brno University of Technology, Department of Telecommunications, Brno, Czech Republic}

\begin{document}

\begin{abstract}

The use of orthogonal projections on high-dimensional input and target data in learning frameworks is studied. First, we investigate the relations between two standard objectives in dimension reduction, preservation of variance and of pairwise relative distances. Investigations of their asymptotic correlation as well as numerical experiments show that a projection does usually not satisfy both objectives at once. In a standard classification problem we determine projections on the input data that balance the objectives and compare subsequent results. Next, we extend our application of orthogonal projections to deep learning tasks and introduce a general framework of augmented target loss functions. These loss functions integrate additional information via transformations and projections of the target data. In two supervised learning problems, clinical image segmentation and music information classification, the application of our proposed augmented target loss functions increase the accuracy.

\end{abstract}

\maketitle
\section{Introduction}
Linear dimension reduction is commonly used for preprocessing of high-dimensional data in complicated learning frameworks to compress and weight important data features. In contrast to nonlinear approaches, the use of orthogonal projections is computationally cheap, since it corresponds to a simple matrix multiplication. Conventional approaches apply specific projections that preserve essential information and complexity within a more compact representation. The projector is usually selected by optimizing distinct objectives, such as information preservation of the sample variance or of pairwise relative distances. Widely used orthogonal projections for dimension reduction are variants of the principal component analysis (PCA) that maximize the variance of the projected data, \cite{pca}. Preservation of relative pairwise distances asks for a near-isometric embedding, and random projections guarantee this embeddings with high probability, cf.~\cite{Dasgupta:2003fk,Baraniuk:2006aa} and see also \cite{Achlioptas:2003wo,Matousek:2008al,Baraniuk:2008fk,Candes:2005vn,Krahmer:2011kx,numax}. The use of random projections is especially favorable for large, high-dimensional data (\cite{Thanei2017}), since the computational complexity is just $O(dkm)$, e.g. using the construction in \cite{Achlioptas:2003wo}, with $d, k \in \N$ being the original and lower dimensions and $m \in \N$ the number of samples. In contrast, PCA needs $O(d^2m)+O(d^3)$ operations (\cite{Golub:1996fk}). Moreover, tasks that do not have all data available at once, e.g. data streaming, ask for dimension reduction methods that are independent of the data. 

In the present manuscript, we study orthogonal projections regarding the interplay between
\begin{itemize}
\item[O1)] preservation of variance,
\item[O2)] preservation of pairwise relative distances,
\end{itemize}
aiming for a sufficient lower-dimensional data representation. We shall consider the Euclidean distance exclusively since it is most widely used in applications, especially for error estimation. On manifolds, the geodesic distance is locally equivalent to the Euclidean distance. The two objectives O1) and O2) are directly addressed by PCA (O1) and random projections (O2). We achieve the following goals: first we clarify mathematically and numerically that the two objectives are competing, i.e. PCA and random projections preserve different kinds of information. Depending on the objectives we discuss beneficial choices of orthogonal projections and numerically find a balancing projector for a given data set. Finally, we define a general framework of augmented target (AT) loss functions for deep neural networks, that integrate information about target characteristics via features and projections. We observe that our proposed methodology can increase the accuracy in two deep learning problems.

\smallskip
In contrast to conventional approaches we study the joint behavior of the two objectives with respect to the entire set of orthogonal projectors. By analyzing the correlation between the variance and pairwise relative distances of projected data, we observe that O1) and O2) are competing and usually cannot be reached at the same time. In numerical learning experiments we investigate heuristic choices of projections applied to input features, for subsequent classification with support vector machine and shallow neural networks. 

In view of learning frameworks, we utilize features and projections on target data. The class of augmented target loss functions incorporates suitable transformations and projections that provide beneficial representations of the target space. It is applied in two supervised deep learning problems dealing with real world data. 

The first experiment is a clinical image segmentation problem in optical coherence tomography (OCT) data of the human retina. Related principles of dimension reduction for other clinical classification problems in OCT have already been successfully applied in \cite{A.-Breger:2017bq}. In the second experiment we aim to categorize musical instruments based on their spectrogram, see \cite{Dorfler:2018xz} for related results. Our utilized augmented target loss functions can increase the accuracy in both experiments. 

\smallskip
The outline is as follows. In Section \ref{sec:2} we address the analysis of the competing objectives and Theorem \ref{theorem} yields the asymptotic correlation between variance and pairwise relative distances of projected data. Section \ref{sec:3} prepares for the numerical investigations by recalling $t$-designs as considered in \cite{Breger:2016vn}, enabling subsequent numerics. Heuristic investigations on projected input used in a straightforward classification task, are presented in Section \ref{sec:4}. Our framework of augmented target loss functions as modified standard loss functions for deep learning, is introduced in Section \ref{general}. Finally, in Sections \ref{sec:appl retina} and \ref{sec:appl music} we present classification experiments on OCT images and musical instruments using aligned augmented target loss functions.



\section{Dimension reduction with orthogonal projections}\label{sec:2}
To reduce the dimension of a high-dimensional data set $x=\{x_i\}_{i=1}^m\subset\R^d$, we map $x$ into a lower-dimensional affine linear subspace $\bar{x}+V$, where $\bar{x}:=\frac{1}{m}\sum_{i=1}^m x_i$ is the sample mean and $V$ is a $k$-dimensional linear subspace of $\R^d$ with $k < d$. This mapping is performed by an orthogonal projector $p\in \G_{k,d}$, where 
\begin{equation*}
\G_{k,d}:= \{p\in\R^{d\times d} : p^2=p,\; p^\top\!=p,\; \rank(p)=k\}
\end{equation*}
denotes the Grassmannian, so that the lower-dimensional data representation is 
\begin{equation}\label{eq:11}
\{\bar{x}+p(x_i-\bar{x})\}_{i=1}^m \subset \bar{x}+V, 
\end{equation}
with $\range(p)=V$. A suitable choice of $p$ within $\G_{k,d}$ depends on further objectives, i.e. which kind of information preservation shall be favored for subsequent analysis tasks. In the following, we consider two objectives associated to popular choices of orthogonal projectors for dimension reduction, in particular, random projectors from $\mathcal{G}_{k,d}$ and PCA. We will first observe that the two objectives are competing, especially in high dimensions, and then discuss consequences.  

\subsection{Objective O1)}
The total sample variance\footnote[1]{We use lower case letters for samples and upper case letters for random vectors/matrices.} $\tvar(x)$ of $x=\{x_i\}_{i=1}^m\subset\R^d$ is the sum of the corrected variances along each dimension, \begin{equation}\label{eq:total variance}
\tvar(x) :=\frac{1}{m-1}\sum_{i=1}^m \|x_i-\bar{x}\|^2.
\end{equation}
PCA aims to construct $p\in\G_{k,d}$, such that the total sample variance of \eqref{eq:11} is maximized among all projectors in $\G_{k,d}$. For other equivalent optimality criteria, we refer to \cite{udell2015thesis}.

The total sample variance of $px=\{px_i\}_{i=1}^m\subset V$ coincides with the one of \eqref{eq:11} and satisfies
\begin{equation*}
\tvar(px)\leq\tvar(x)
\end{equation*}
for all $p \in \G_{k,d}$. Thus, PCA achieves optimal variance preservation. The total variance \eqref{eq:total variance} can also be expressed via pairwise absolute distances
\begin{align}
\tvar(x) &=\frac{1}{m(m-1)} \sum_{i<j} \norm{x_i - x_j}^2 \label{tvarx}. 
\intertext{Equally, it holds that} 
\tvar(px) &= \frac{1}{m(m-1)} \sum_{i<j} \norm{p(x_i) - p(x_j)}^2 \label{tvarpx}, 
\end{align}
which reveals that PCA maximizes the sample mean of the projected pairwise absolute distances. 

\subsection{Objective O2)}\label{JLsec}
In contrast to pairwise absolute distances, the Johnson-Lindenstrauss Lemma targets the global property of preservation of pairwise relative distances:
\begin{lemma}[Johnson-Lindenstrauss, cf.~\cite{Dasgupta:2003fk,Matousek:2008al}]\label{lemma JL}
For any $0 < \epsilon < 1$, any $k \leq d, m \in\N$,  with 
\begin{equation*}
\frac{4 \log(m)}{\epsilon^2/2 -\epsilon^3/3}\leq k,
\end{equation*}
and any set $\{x_i\}_{i=1}^m\subset\R^d$, there is a projector $p\in\G_{k,d}$ such that
\begin{equation}\label{JLeq}
(1 - \epsilon) \norm{x_i- x_j}^2 \leq \tfrac{d}{k}\norm{p(x_i) - p(x_j)}^2 \leq (1 + \epsilon) \norm{x_i- x_j}^2 
\end{equation}
holds for all $i<j$.
\end{lemma}
For small $\epsilon > 0$, the projector $p$ in Lemma \ref{lemma JL} yields that all of the $\frac{m(m-1)}{2}$ pairwise relative distances 
\begin{equation}\label{eq:set distances}
\Big\{\frac{d}{k}\frac{\|p(x_i) - p(x_j)\|^2}{\|x_i - x_j\|^2}:i<j\Big\}
\end{equation}
are close to $1$, i.e. the projection $p$ preserves all scaled pairwise relative distances well. A good choice of $p$ in Lemma \ref{lemma JL} is based on random projectors $P\sim \lambda_{k,d}$, where $\lambda_{k,d}$ denotes the 
unique orthogonally invariant probability measure on $\G_{k,d}$. The following Theorem is essentially proved by following the lines of the proof of Lemma \ref{lemma JL} in \cite{Dasgupta:2003fk} after replacing the constant $4$ with $(2 + \tau)2$ in the respective bound on $k$.  
\begin{theorem}\label{allorth}
For any $0 < \epsilon < 1$, any $k \leq d, m \in\N$ and any $0<\tau$ with 
\begin{equation*}
\frac{(2 + \tau) 2\log(m)}{\epsilon^2/2 - \epsilon^3/3} \leq k,
\end{equation*}
and any set $\{x_i\}_{i=1}^m\subset\R^d$, the random projector $P\sim\lambda_{k,d}$ satisfies 
\begin{equation}\label{eq:JL random}
\Big\{\frac{d}{k}\frac{\|P(x_i) -P(x_j)\|^2}{\|x_i - x_j\|^2}:i<j\Big\}\in [1-\epsilon,1+\epsilon]
\end{equation}
with probability at least $1 - \tfrac{1}{m^{\tau}} + \tfrac{1}{m^{\tau+1}}$. 
\end{theorem}

\subsection{Competing objectives}\label{comp}
A projector $p$ satisfying the near-isometry property \eqref{JLeq} implies
\begin{equation*}
(1 - \epsilon) \tfrac{k}{d} \tvar(x) \leq \tvar(px) \leq (1 + \epsilon) \tfrac{k}{d} \tvar(x),
\end{equation*}
so that the total variance of the projected data $px$ may not be maximized for $k<d$. In particular, with high probability a random projector $P\sim\lambda_{k,d}$ does not suit the objective of maximizing the total variance, and we even observe $\mathbb{E}\tvar(Px) = \frac{k}{d}\tvar(x)$, see \eqref{exp1} in the appendix. PCA does not guarantee any local geometric property and distances between pairs of points can be arbitrarily distorted \cite{Achlioptas:2003wo}, see \cite{Neumayer:2019sy} for more robust PCA. The preservation of larger distances is favored since PCA maximizes \eqref{tvarpx} among all $p \in \G_{k,d}$ and $\|p(x_i)-p(x_j)\| \leq \| x_i - x_j \|$ holds for all $i<j$. Close but distinct points, could even be projected onto a single point, which violates the preservation of pairwise relative distances, see Figure \ref{pcaplot}. 

\begin{figure}
\includegraphics[width = 0.48\textwidth]{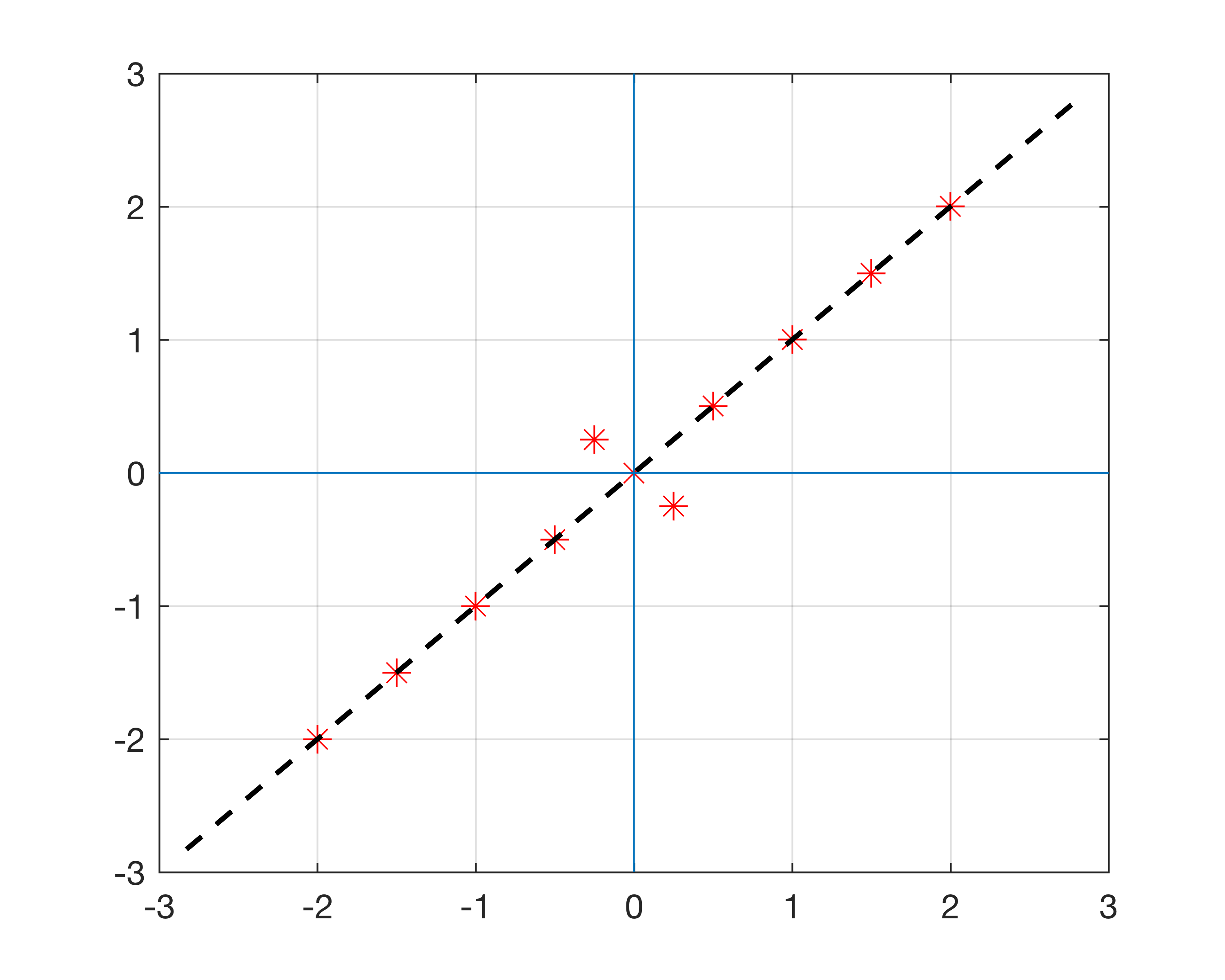}
\caption{A trivial example of PCA distorting smaller distances. Choosing the first principal component, PCA projects the two dimensional data points {\color{red}*} onto the plane of the first eigendirection (- -). The Euclidian distances of the points lying on the diagonal are preserved, whereas the two points with smaller distances are projected onto a single point (the origin).}
\label{pcaplot}
\end{figure}

To more quantitatively understand the relation between the two competing objectives, we consider the sample mean and the uncorrected sample variance of the pairwise relative distances \eqref{eq:set distances},
\begin{align}
 \mathcal{M}(p,x)& := \frac{2}{m(m-1)} \sum_{i<j}\frac{d}{k}\frac{\|p(x_i - x_j)\|^2}{\|x_i-x_j\|^2}, \label{mpx} \\
 \mathcal{V}(p,x)&:=\frac{2}{m(m-1)}\sum_{i<j}\frac{d^2}{k^2}\frac{\|p(x_i - x_j)\|^4}{\|x_i-x_j\|^4} - \mathcal{M}(p,x)^2.\label{dist}
 \end{align}
Recall that good preservation of the relative pairwise distances in \eqref{eq:set distances} asks for $\mathcal{M}(p,x)$ being close to $1$ and the variance $\mathcal{V}(p,x)$ being small. In the following, we analyze $\tvar(px)$, $\mathcal{M}(p,x)$, and $\mathcal{V}(p,x)$ and their expectations for random $P\in\G_{k,d}$.

\begin{figure}
\subfigure[$k = 10$]{\includegraphics[width = 0.48\textwidth]{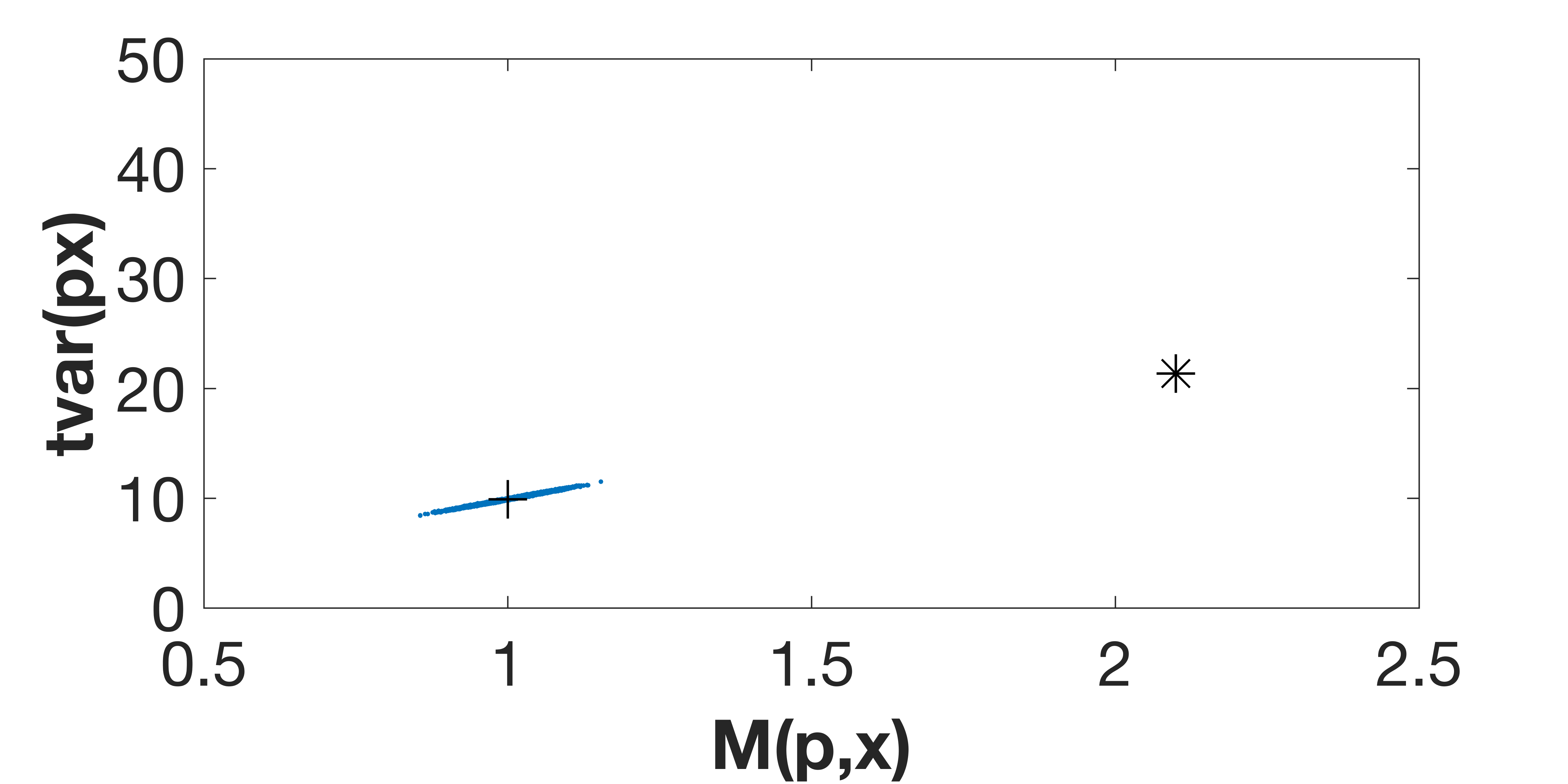}}
\subfigure[$k = 20$]{\includegraphics[width = 0.48\textwidth]{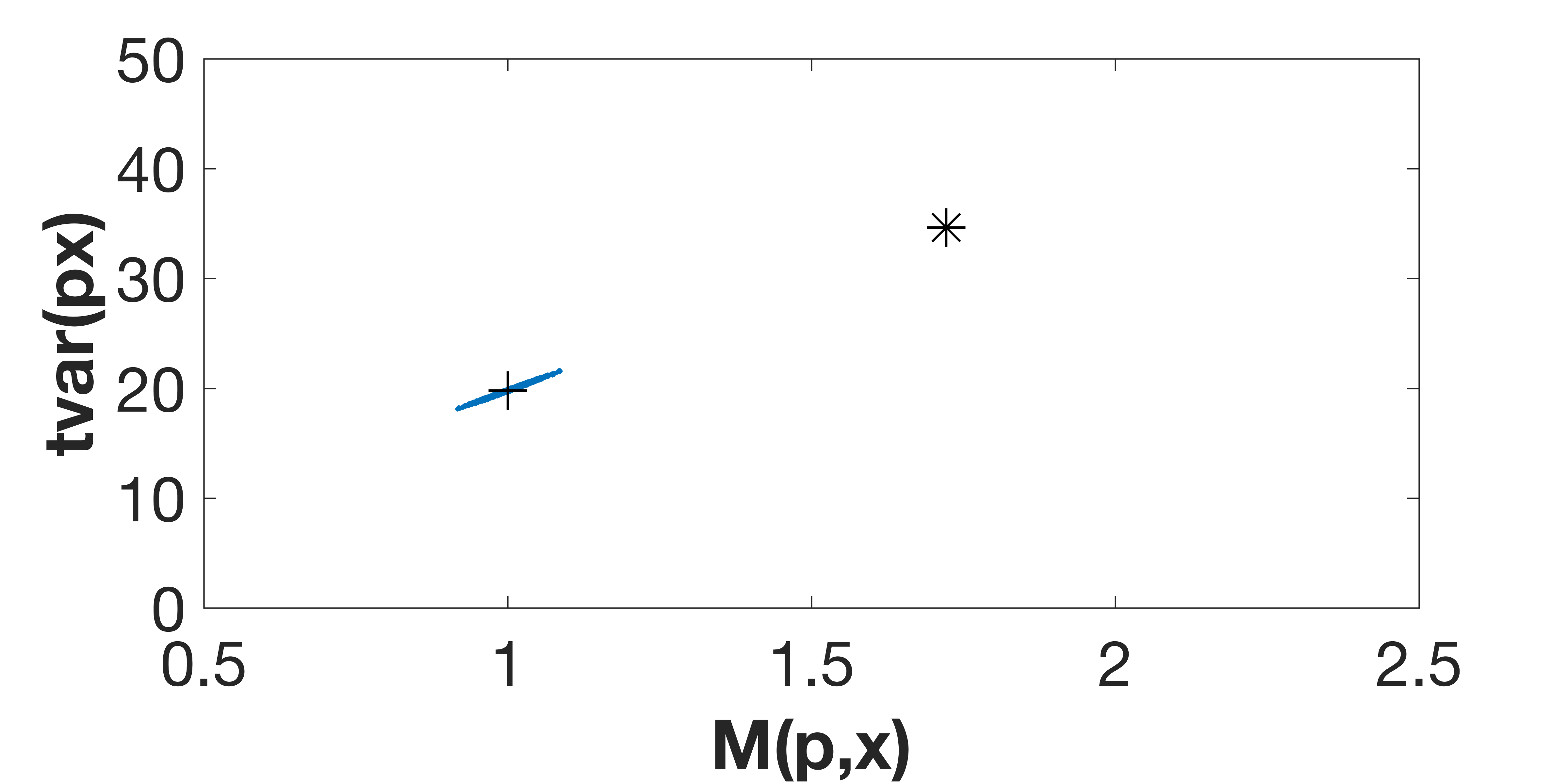}}
\subfigure[$k = 30$]{\includegraphics[width = 0.48\textwidth]{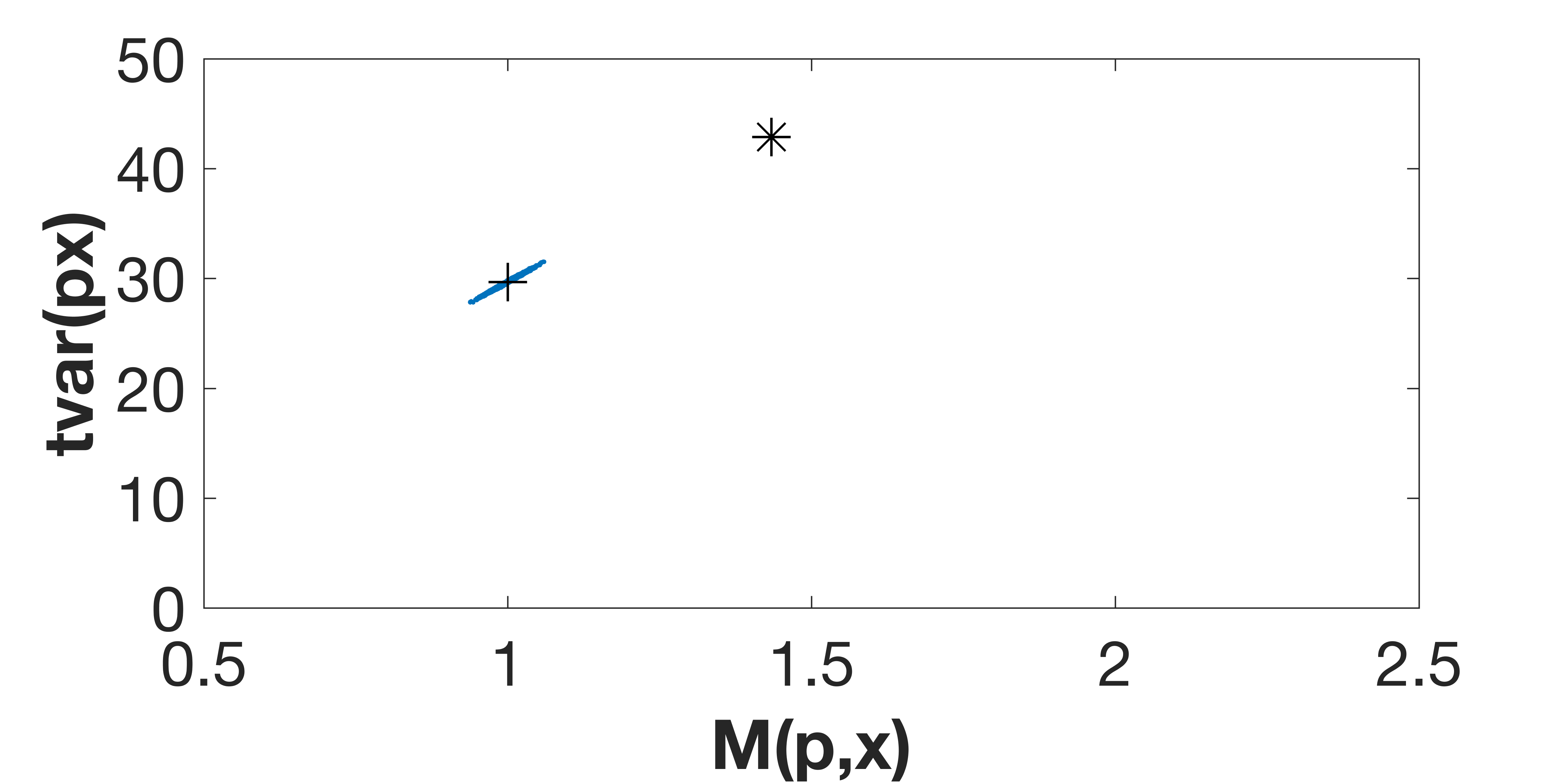}}
\subfigure[$k = 40$]{\includegraphics[width = 0.48\textwidth]{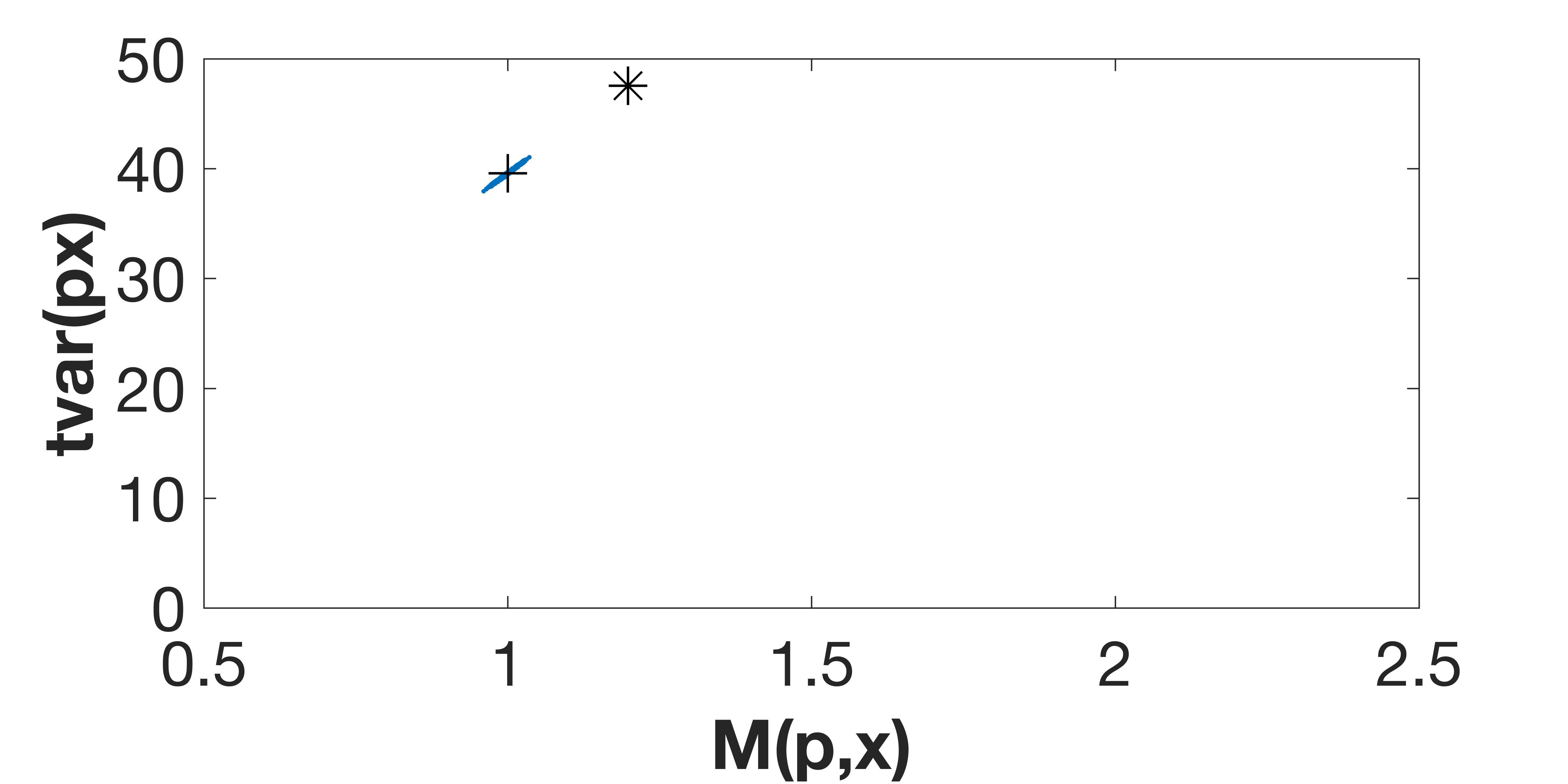}}
\subfigure[$k = 10$]{\includegraphics[width = 0.48\textwidth]{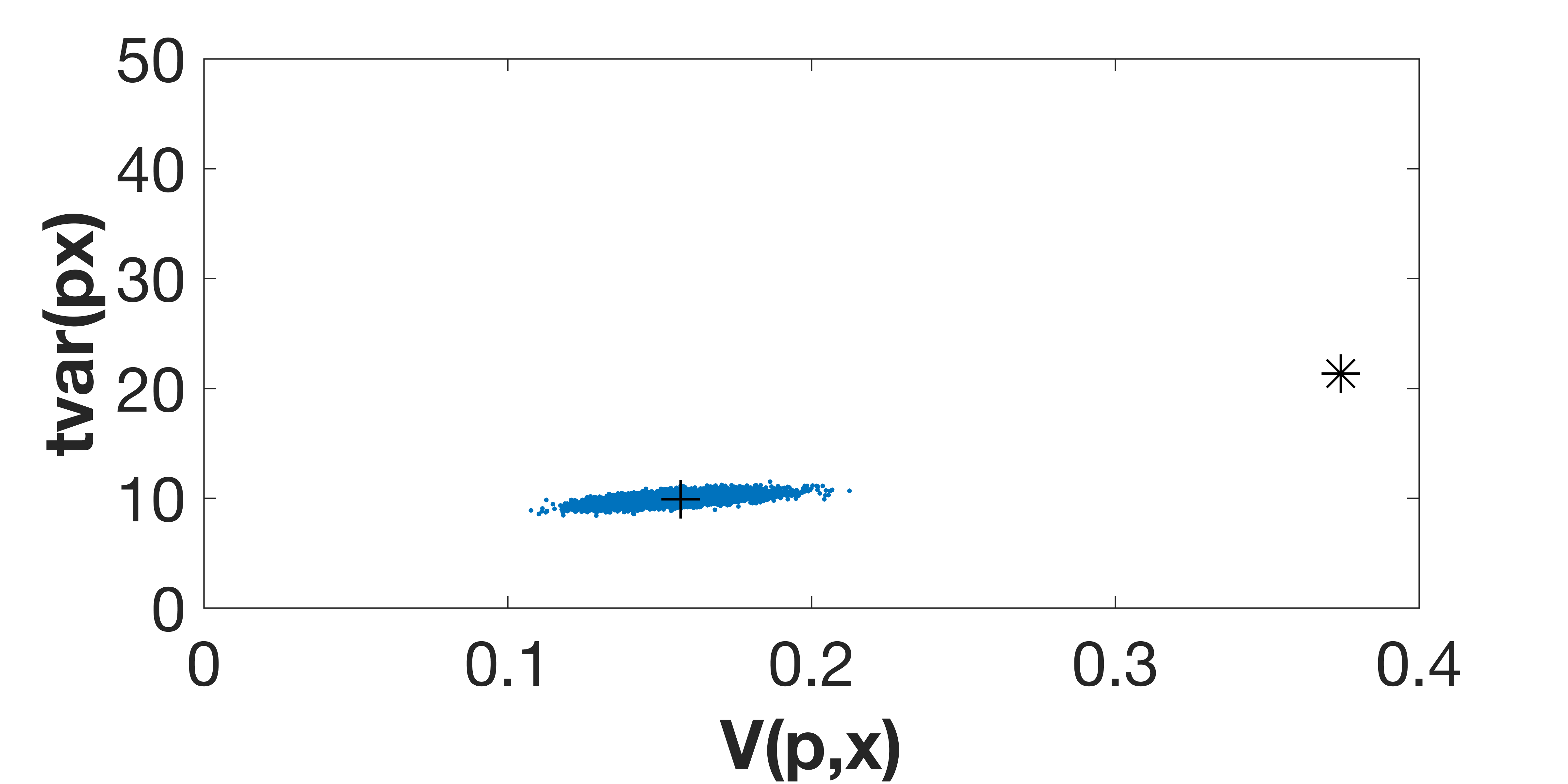}}
\subfigure[$k = 20$]{\includegraphics[width = 0.48\textwidth]{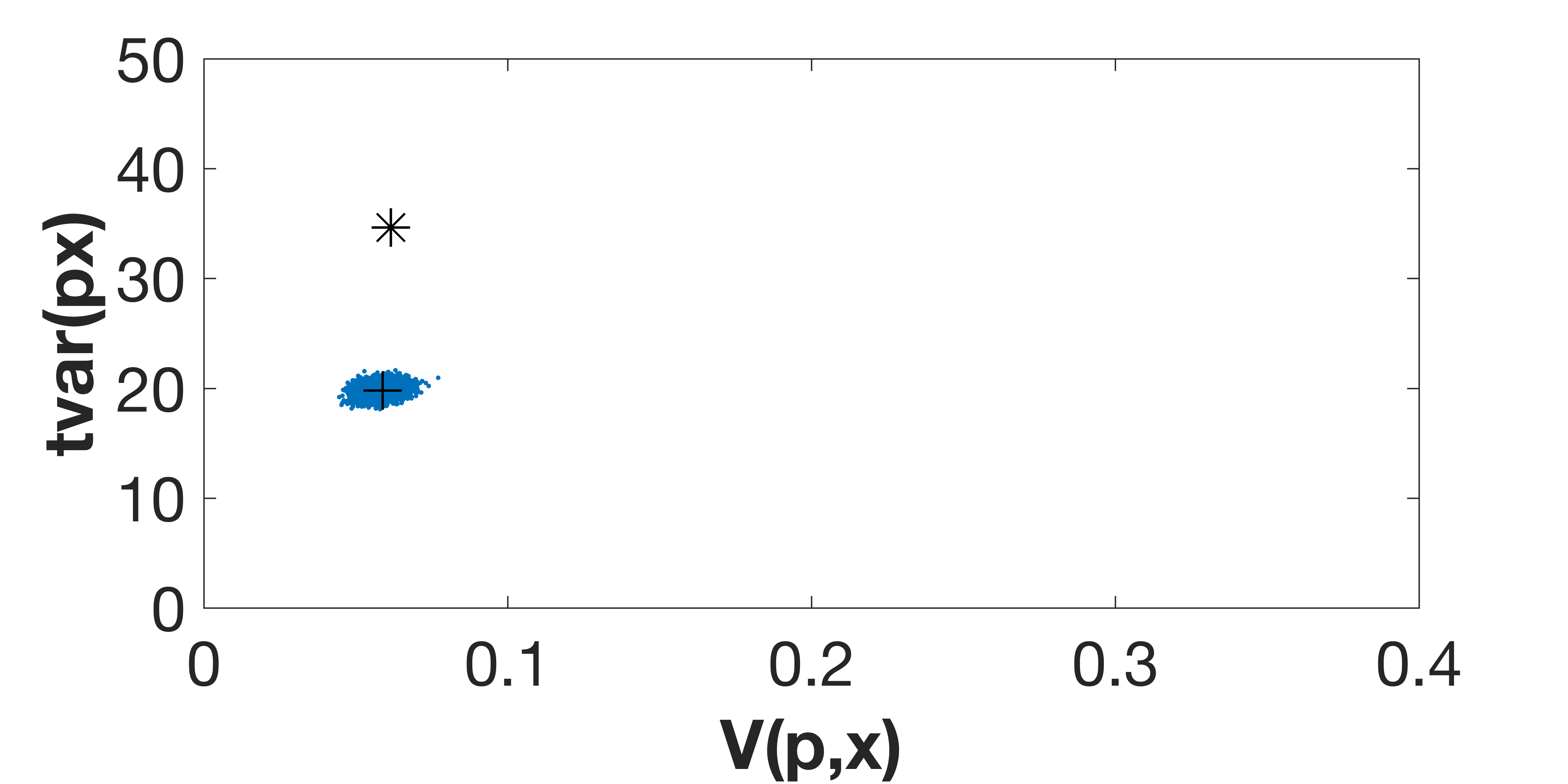}}
\subfigure[$k = 30$]{\includegraphics[width = 0.48\textwidth]{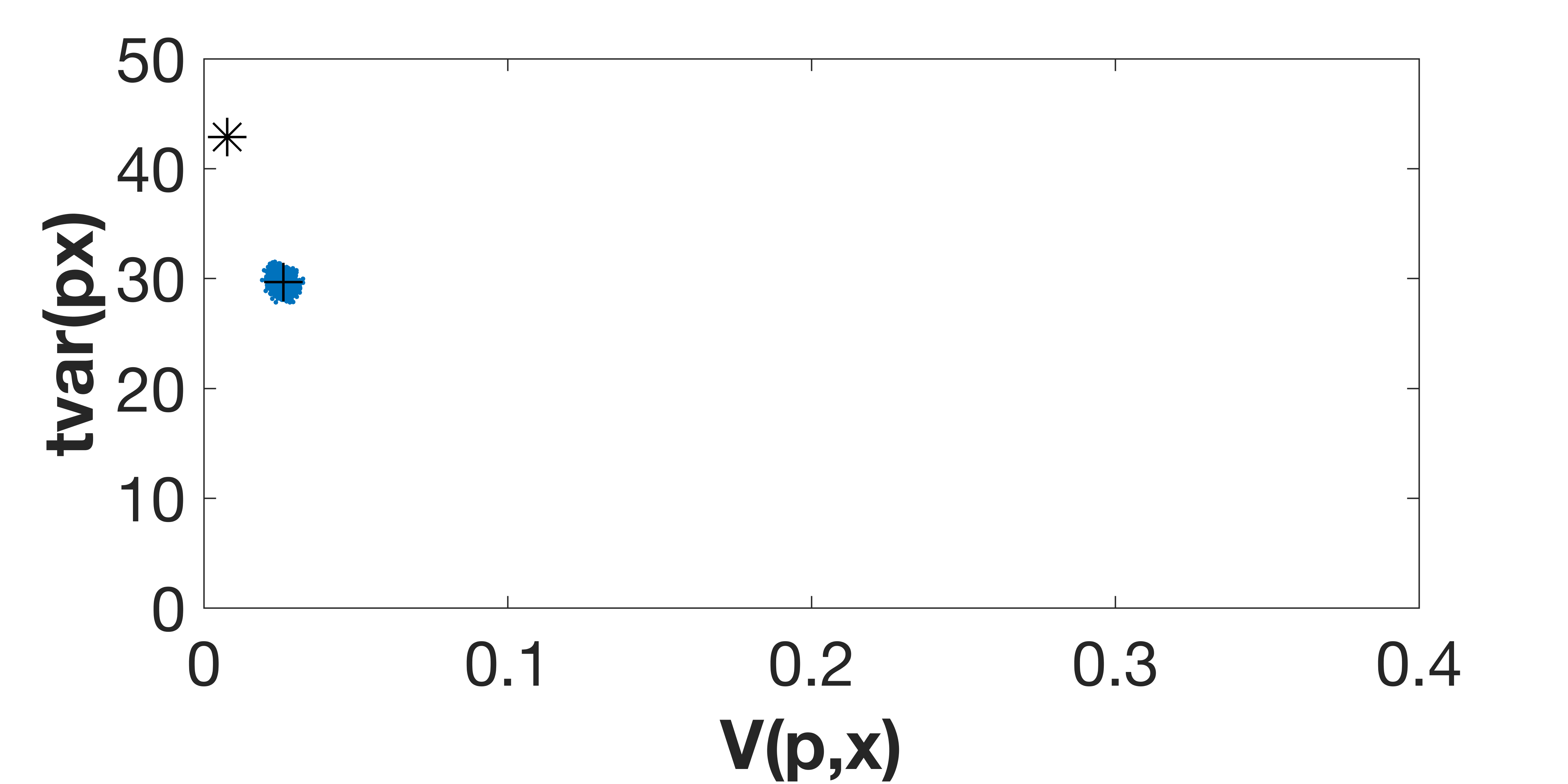}}
\subfigure[$k = 40$]{\includegraphics[width = 0.48\textwidth]{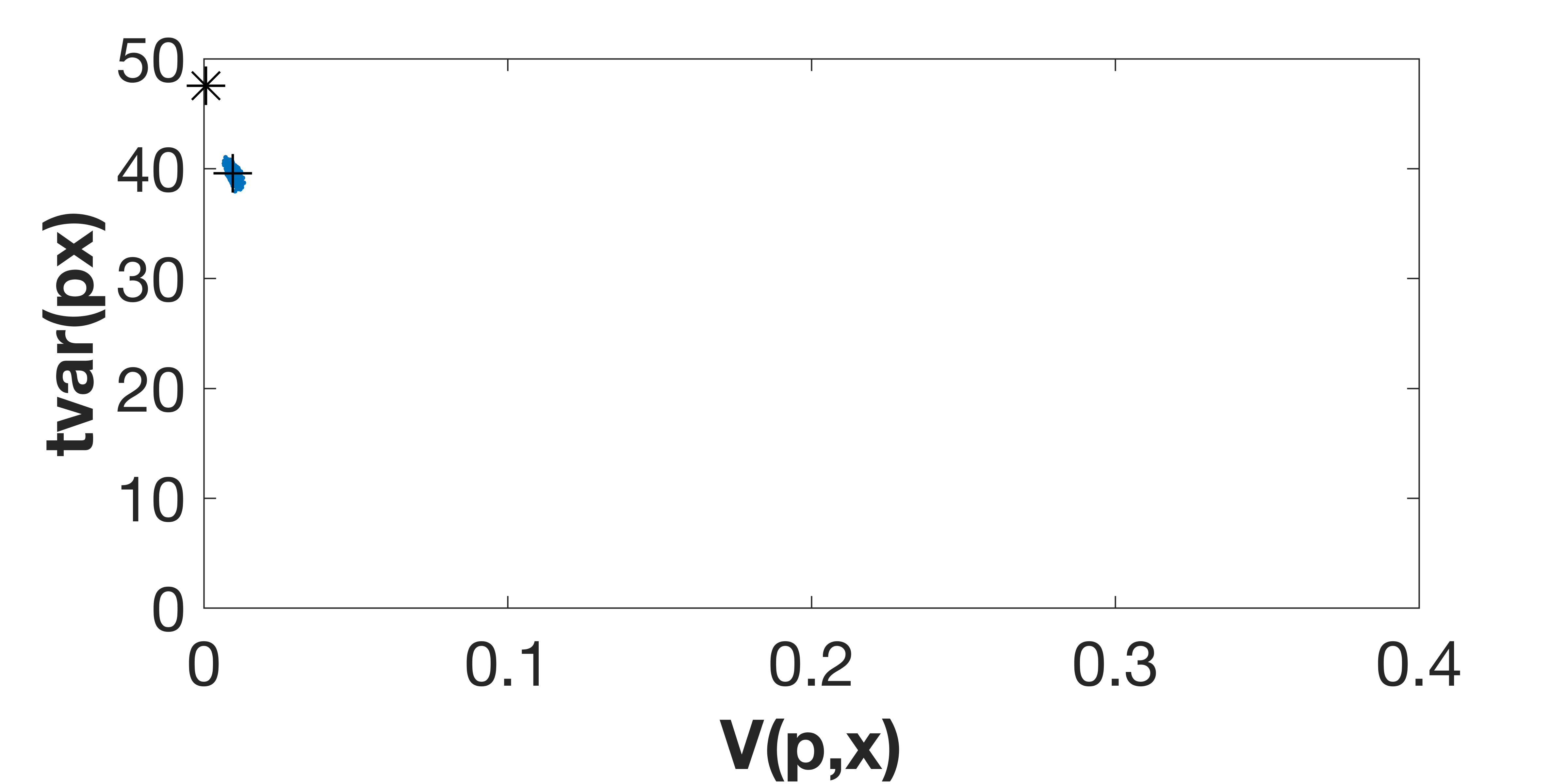}}
\caption{Competing properties: $10000$ random projections $p \sim \lambda_{k,50}$ versus PCA ($*$), plotted concerning $\tvar(px)$, $\mathcal{M}(p,x)$ and $\mathcal{V}(p,x)$. The normal distributed fixed data set $x$ has total variance $\tvar(x) = 49.5$. Random projections cluster around their expectation values \eqref{exptvar}, \eqref{expM} and \eqref{eq:sam}, marked by \small{+}.}\label{figcomprop}
\end{figure}

In Figure \ref{figcomprop} we see a simple numerical experiment, where we first create an independent, normally distributed fixed data set $\{x_i\}_{i=1}^{m}$ with $x_i \in \R^d$ for $i=1, \dots, m$ and $m = 100$, $d = 50$. We then compute PCA, for $k = 10, 20, 30, 40$, as well as $n = 10000$ random projections $p$ distributed according to $\lambda_{k,50}$. In Figure \ref{figcomprop} (a) - (d) we can see that the more $k$ differs from $d$, the more PCA and random projections differ concerning $\tvar(px)$ and $\mathcal{M}(p,x)$. Those differences may lead diverse behavior in subsequent data analysis. Moreover, we compare $\mathcal{M}(p,x)$ and $\mathcal{V}(p,x)$ in Figure \ref{figcomprop} (e) - (h) for the different $k$. We can see that again when $k$ is much smaller than $d$, random projections and PCA differ more concerning the variance of pairwise distances $\mathcal{V}(p,x)$. For $k=10$ the variance for PCA is higher in comparison to random projections, see Figure \ref{figcomprop} (e), for $k = 40$ vice versa, see Figure \ref{figcomprop} (h). Note that the theoretical bounds stated in Theorem \ref{lemma JL} are much higher than the dimensions $k$ used in the experiments, but the projections still preserve relative pairwise distances very well. In \cite{Bingham:2001:RPD:502512.502546} similar observations were made on empiric experiments with image and text data. 

The amount of variance kept in the principal components comparing real world and random data has been experimentally studied, e.g. in \cite{comprealrand2} and \cite{comprealrand}. Both studies determine that the difference occurs mainly in the first principal component.

\begin{remark}
In the numerical example we compare random projections and PCA directly, serving as the corresponding projections to the objectives O1) and O2). We observe that even for not so high dimensional ($d = 50$) data $x$ and $k\ll d/2$, PCA severely looses information in terms of total variance, i.e. more than $50 \%$ for $k=10$, and more importantly, looses much more information on pairwise relative distances than random projections. If both types of information are of interest, pairwise relative distances and high total variance, one should therefore favor random projections over PCA for $k \ll d/2$ to balance the two objectives O1) and O2) and vice versa. Note that with a large amount of data one might still want to favor random projectors since their construction is computationally much cheaper and independent from the data. On the other hand, if objective O2) is negligible, e.g. tasks with very noisy data, then PCA would be the favorable choice for all $k$.
\end{remark}

Information of data can be quantified and expressed in different ways. One crucial part in dimension reduction is the decision of what kind of information shall be kept, which depends on several parameters including the quality of the data and the analysis task. Variants of PCA, focusing on the preservation of variance, have been widely used in real world problems with big success, especially in denoising, when the preservation of all pairwise relative distances may be counterproductive, e.g. in dMRI imaging \cite{VERAART2016394} and color filter array images \cite{4798177}. Drawbacks are the necessity for all data being available from the start and the high computational costs. For very high dimensional and large data sets the computation of PCA is often not feasible. Besides the huge benefit of data independence and low computational cost when using random projections, the near-isometry property often allows to establish that the solution found in the low-dimensional space is a good approximation to the solution in the original space (\cite{Achlioptas:2003wo}, \cite{MAL-035}). 

Algorithms in machine learning often need or benefit from sufficient estimates of pairwise distances, e.g. approximate nearest-neighbor problems, supervised classification \cite{numax} and subspace clustering \cite{subclus}. In \cite{Linial1995} algorithmic applications of near-isometry embeddings have been introduced. In \cite{Bingham:2001:RPD:502512.502546} random projections have been successfully applied to noisy and noiseless text and image data. The experimental studies include the comparison of preservation of pairwise distances between random projections and PCA. The results coincide with our observations, that for $k > d/2$ PCA is able to preserve the pairwise distances sufficiently, whereas for $k < d/2$ PCA distorts them. The smaller $k$ the worse the distortion, whereas random projections preserve similarities still well for very small $k$, while being computationally much cheaper than PCA. One should point out again that favoring preservation of pairwise distances relies on the accuracy of the original distances.  

PCA and random projections are orthogonal projections favoring two different aims. We want to study in the context of the whole set of orthogonal projections if the two objectives O1) and O2) could be reached at the same time. We will see that the objectives act competing and therefore we suggest a balancing projector for tasks that benefit from both objectives.
\subsection{Covariances and correlation between competing objectives}\label{correl}
For further mathematical analysis we first introduce a more general class of probability measures on $\G_{k,d}$ that resemble $\lambda_{k,d}$ sufficiently well:
\begin{definition}\label{tdes}
A Borel probability measure $\lambda$ on $\G_{k,d}$ is called a \emph{cubature measure of strength $t$} if 
 \begin{equation*}
    \int_{\G_{k,d}} f(p)\mathrm d\lambda_{k,d}(p) = \int_{\G_{k,d}} f(p)\mathrm d\lambda(p),\quad \text{for all } f\in \Pol_t(\R^{d^2}),
  \end{equation*}
  where $\Pol_t(\R^{d^2})$ denotes the set of multivariate polynomials of total degree $t$ in $d^2$ variables. 
\end{definition}
Existence of cubature measures is studied, for instance, in \cite{Harpe:2005fk}. For random $P$, we now determine the expectation values for our $3$ quantities of interest: $\tvar(Px)$, $\mathcal{M}(P,x)$, and $\mathcal{V}(P,x)$. If $P\sim\lambda$ and $\lambda$ is a cubature measure of strength at least $2$, the identities \eqref{exp1} and \eqref{exp2} in the appendix and a short calculation yield 
\begin{align}
&\mathbb{E}\tvar(Px) = \tfrac{k}{d}\tvar(x),\label{exptvar}\\
&\mathbb{E}\mathcal{M}(P,x)  =1, \label{expM}\\
&\mathbb{E}\mathcal{V}(P,x) = a_{k,d} (1-\tfrac{4}{m^2(m-1)^2}\sum_{\substack{i<j\\
l<r}}\langle \tfrac{x_i - x_j}{\|x_i-x_j\|},\tfrac{x_l- x_r}{\|x_l-x_r\|}\rangle^2),\label{eq:sam}
\end{align}
where $a_{k,d} =  \tfrac{2d(d-k)}{k(d-1)(d+2)}$. The expected sample variance in \eqref{eq:sam} satisfies
\begin{equation*}
\mathbb{E}\mathcal{V}(P,x) \leq a_{k,d} \longrightarrow \frac{2}{k},\quad\text{for}\quad d\rightarrow\infty.
\end{equation*}
This asymptotic bound relates to Theorem \ref{allorth} and alludes to a near-isometry property of the type \eqref{eq:JL random} for $k$ sufficiently large.

The following Theorem \ref{theorem} provides a lower bound  for random $P$ on the population correlation 
  \begin{equation}\label{eq:bound from below cov}
  \Corr(\mathcal{M}(P,x), \tvar(Px)) = \frac{\Cov(\mathcal{M}(P,x) , \tvar(Px))}{\sqrt{\Var(\mathcal{M}(P,x))} \sqrt{\Var(\tvar(Px))}}.
 \end{equation} 
It holds for arbitrary dimensions $d$ and subsequently specifies the asymptotic behavior for $d\rightarrow\infty$:
\begin{theorem}\label{theorem}
Let $x=\{x_i\}_{i=1}^m\subset\R^d$ be pairwise different and let $P\sim\lambda$, with $\lambda$ being a cubature measure of strength at least $2$. For $d \geq \tfrac{m(m-1)}{2}$, the correlation \eqref{eq:bound from below cov} is bounded from below by 
\begin{equation}\label{thm1}
 \tfrac{\min_{i\neq j} \norm{x_i - x_j }^2}{\max_{i\neq j} \norm{x_i - x_j }^2}  - \tfrac{m(m-1)}{2d} \cdot \tfrac{\max_{i\neq j} \norm{x_i - x_j }^2}{\min_{i\neq j} \norm{x_i - x_j }^2}. 
\end{equation}
Let $\{x_i\}_{i=1}^m\subset \R^d$ be random points, whose entries are independent, identically distributed with finite $4$-th moments, that are uniformly bounded in $d$. Then \eqref{thm1} converges towards $1$ in probability for $d\rightarrow \infty$.
\end{theorem}

\begin{figure}
\subfigure[$d = 50$]{\includegraphics[width = 0.48\textwidth]{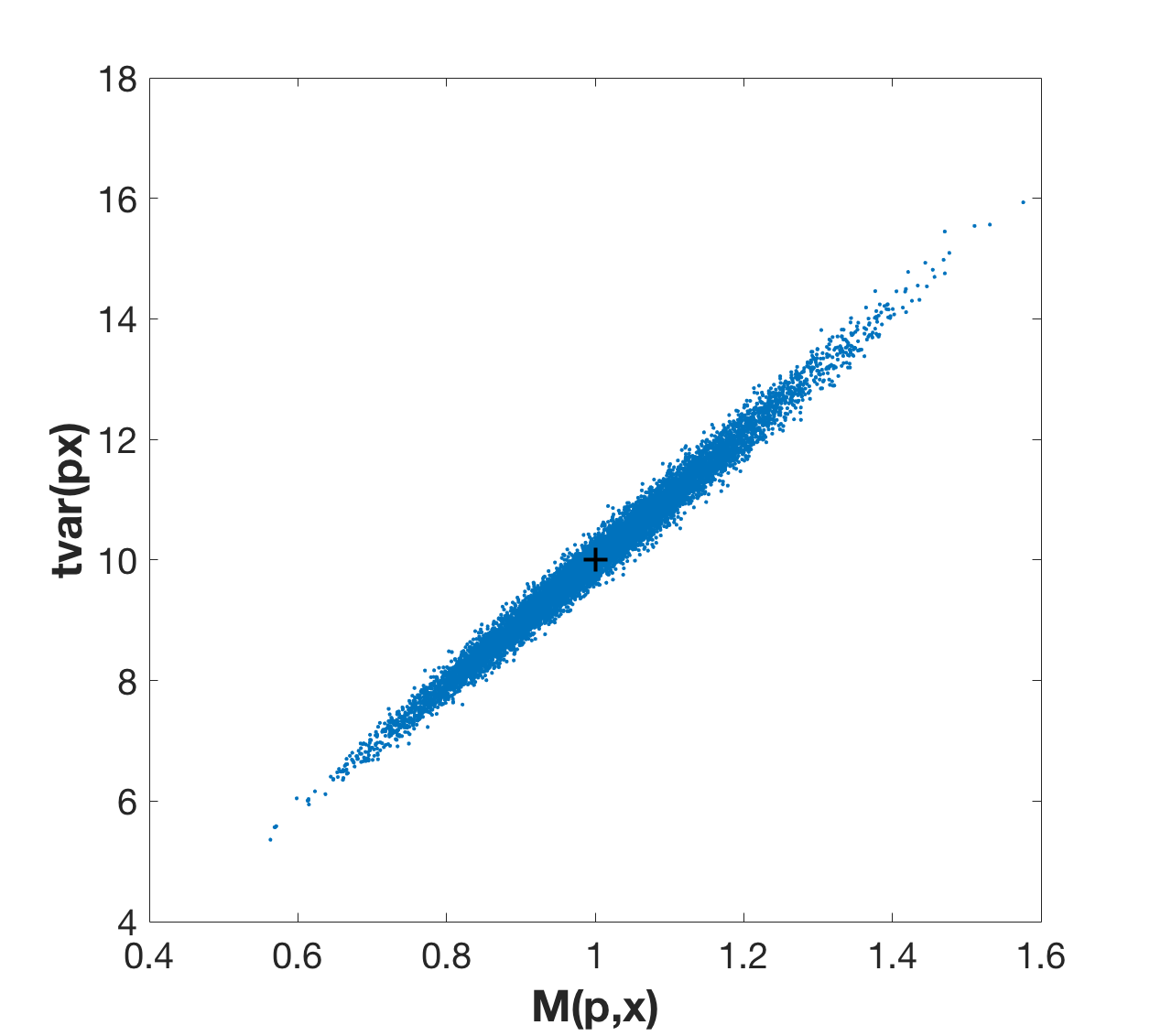}}
\subfigure[$d = 100$]{\includegraphics[width = 0.48\textwidth]{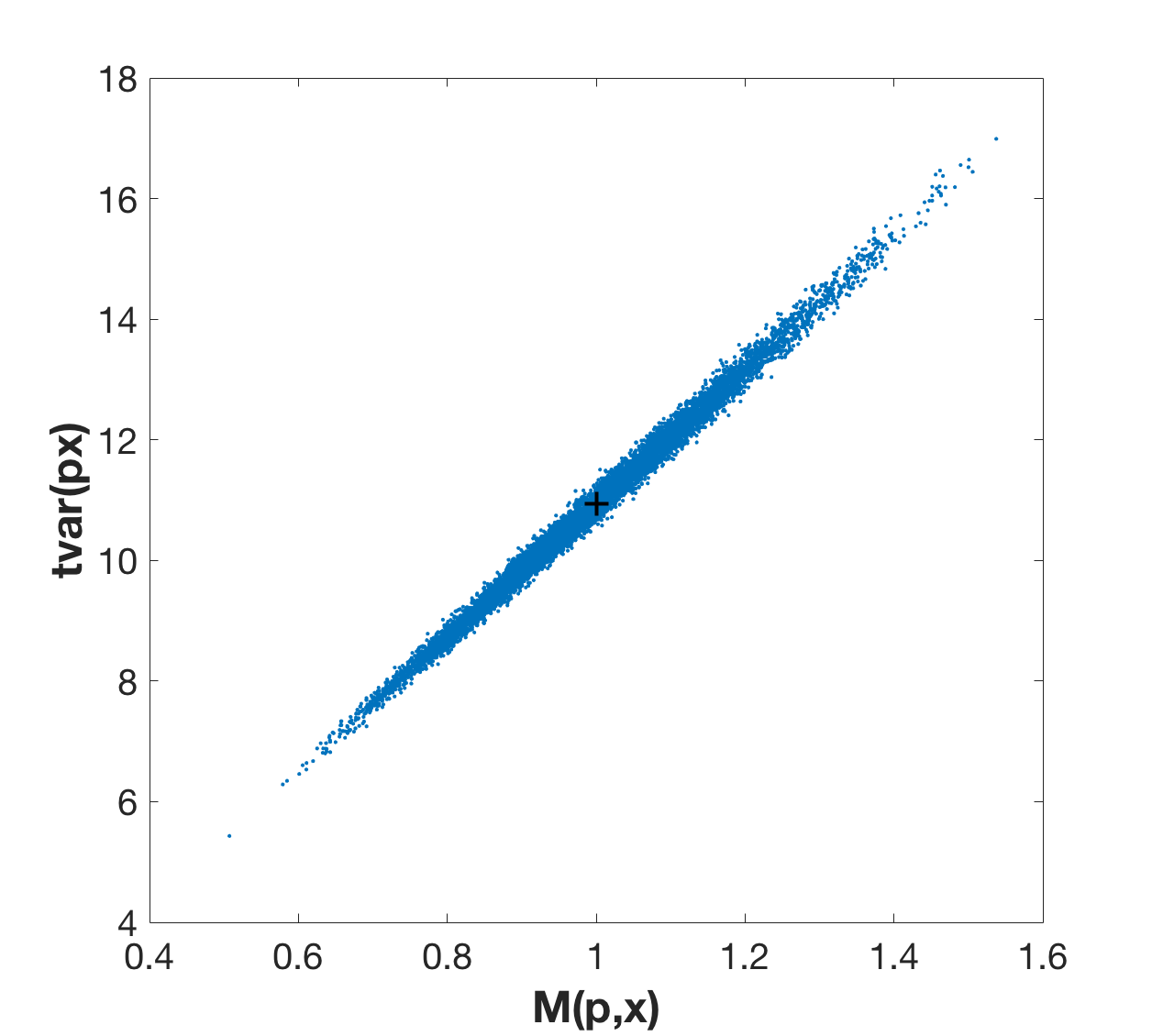}}
\subfigure[$d = 200$]{\includegraphics[width = 0.48\textwidth]{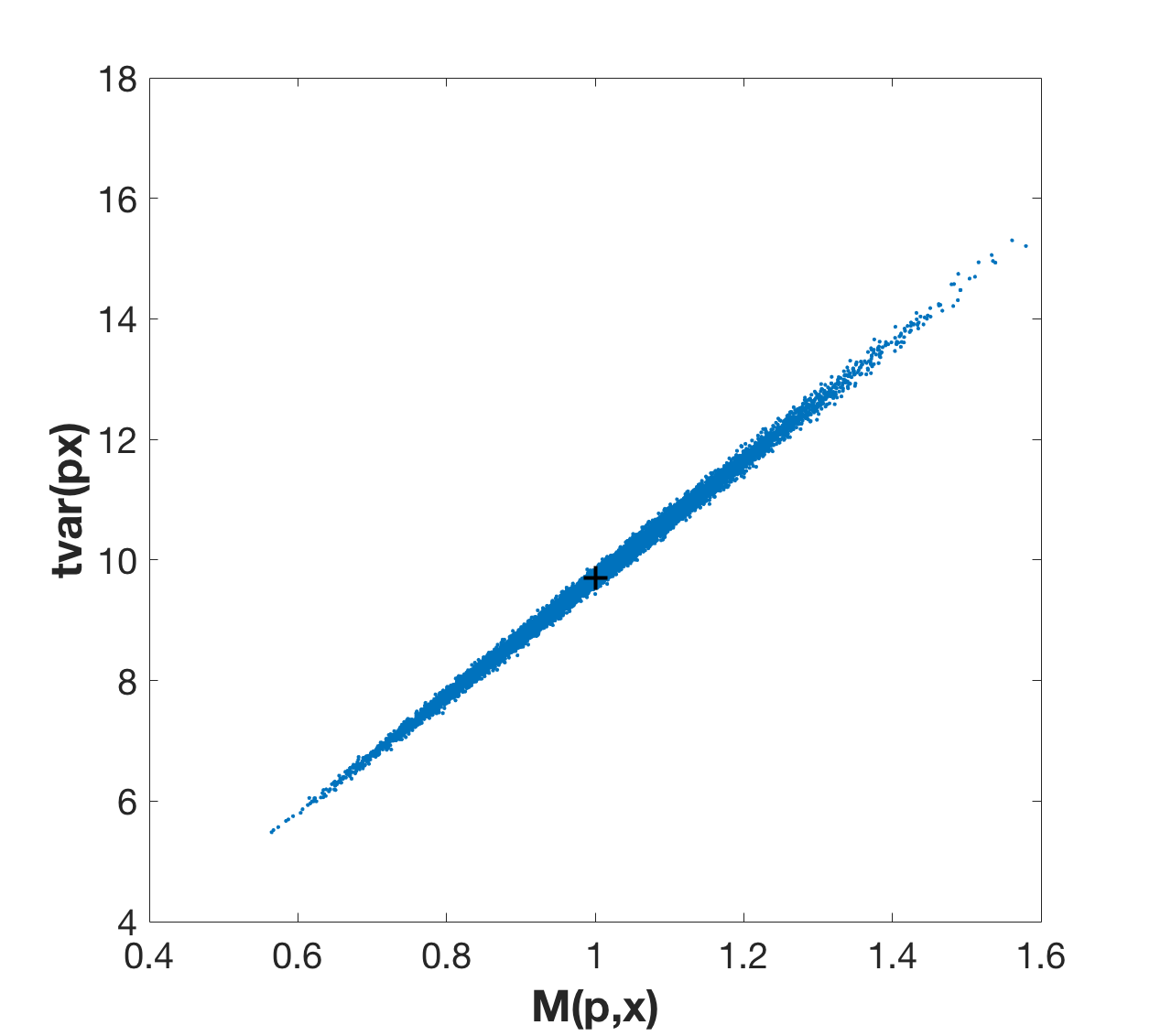}}
\subfigure[$d = 500$]{\includegraphics[width = 0.48\textwidth]{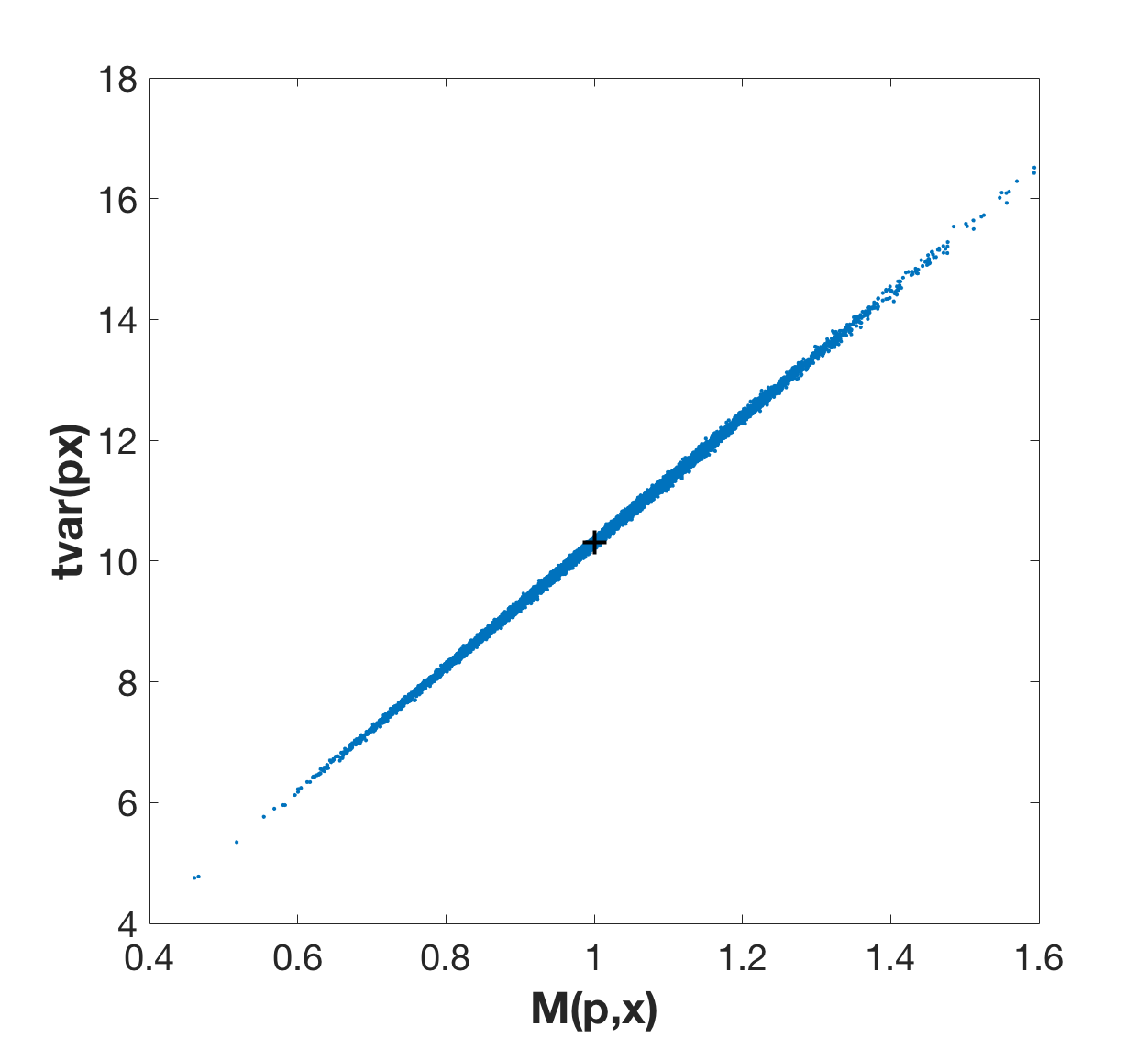}}
\caption{ For $x=\{x_i\}_{i=1}^{10}\subset\R^d$ with independent, normal distributed entries, we independently sample $10 000$ random projectors $p$ from $\lambda_{10,d}$ and plot $\mathcal{M}(p,x)$ versus $\tvar(px)$. The expectation values with respect to $P \sim \lambda$ are marked with \small{+}. \normalsize The correlation is already $0.9916$ for $d = 50$ and grows further when $d$ increases, namely with values $0.9961, 0.9985, 0.9996$ for $d = 100,200,500$.}\label{figthm}
\end{figure}


The strong correlation for large dimensions $d$ in the second part of Theorem \ref{theorem} suggests that increasing $\tvar(Px)$ may also lead to increasing $\mathcal{M}(P,x)$, see Figure \ref{figthm} for illustration. Thus, large projected total variance $\tvar(Px)$ and the preservation of scaled pairwise distances, i.e. $\mathcal{M}(P,x)$ being close to $1$, are competing properties. As discussed in Section \ref{comp}, the choice of which kind of information is favorable to preserve, depends on the data and the task; e.g. denoising (O1) and nearest neighbor classification (O2). PCA and random projections are extreme in preserving either O1) or O2). We will heuristically study the behavior of orthogonal projections balancing both objectives in the next section and will state a numerical experiment where a balancing projector yields highest classification accuracy.

\begin{remark}
The second part of Theorem \ref{theorem} relates to the well-known fact that random vectors in high dimensions are almost orthogonal, \cite{Ball:1997dw}, and standard concentration of measure arguments may lead to more quantitative statements, cf.~\cite{Vershynin:2012fk}.
\end{remark}

\section{Preparations for numerical experiments}\label{sec:3}
For the numerical experiments we need finite sets of projectors that represent the overall space well, i.e.~cover $\G_{k,d}$ properly.
\subsection{Optimal covering sequences} \label{optcov}
Let the \emph{covering radius} of a set $\{p_l\}_{l=1}^n\subset \G_{k,d}$ be denoted by
\begin{equation}\label{eq:def cov}
\varrho(\{p_l\}_{l=1}^n):=\sup_{p\in \G_{k,d}} \min_{1\leq l\leq n} \|p-p_l\|_{\F},
\end{equation}
where $\|\cdot\|_{\F}$ is the Frobenius norm. The smaller the covering radius, the better the set $\{p_l\}_{l=1}^n$ represents the entire space $\G_{k,d}$. I.e., there are smaller holes and the points $\{p_l\}_{l=1}^n$ are better distributed within $\G_{k,d}$. Following Lemma \ref{lemma JL} we can connect finite sets of projections and their covering radius to the near-isometry property:
\begin{lemma}\label{GJL}
Let $\{p_l\}_{l=1}^{n}\subset \G_{k,d}$ and denote $\varrho:=\varrho(\{p_l\}_{l=1}^n)$. For any $0 < \epsilon < 1$, any $m,k,d\in\N$ with 
\begin{equation*}
\frac{4 \log(m)}{\epsilon^2/2 -\epsilon^3/3}\leq k\leq d,
\end{equation*}
and any $ \{x_i\}_{i=1}^m\subset\R^d$, there is $l_0\in\{1,\ldots,n\}$ such that
\begin{equation}\label{covjl}
(1 - \delta) \norm{x_i - x_j}^2 \leq \tfrac{d}{k}\norm{p_{l_0}(x_i) - p_{l_0}(x_j)}^2 \leq (1 + \delta) \norm{x_i -x_j}^2, i<j,
\end{equation}
where $\delta = \epsilon + 2\varrho \sqrt{\frac{(1 + \epsilon)d }{k}} + \frac{d }{k}\varrho^2$.
\end{lemma}
\begin{proof}
Given an arbitrary projector $p\in\G_{k,d}$, there is an index $l_0 \in \{1, \dots, n\}$ such that 
\begin{equation*}
 \|p_{l_0}x - px \|\leq \|p_{l_0} - p \|_{\F} \|x\|\leq \varrho \|x\|,\quad x\in\R^d.
\end{equation*}
From here, standard computations imply Lemma \ref{GJL}. We omit the details.
\end{proof}
The accuracy of the near-isometry property in \eqref{covjl} depends on the covering radius. Therefore, a set $\{p_l\}_{l=1}^n\in\G_{k,d}$ with a small covering radius $\varrho$ is more likely to contain a projector with better preservation of pairwise relative distances. According to \cite{Breger:2016rc}, it holds that \footnote{We use the symbols 
  $\lesssim$ and $\gtrsim$ to indicate that the corresponding inequalities hold up to a positive constant factor on the respective right-hand side. The notation $\asymp$ means that both relations $\lesssim$ and $\gtrsim$ hold.} $\varrho\gtrsim n^{-\frac{1}{k(d-k)}}$, and we shall see next, how to achieve this lower bound. 

A set of projectors $\{p_l\}_{l=1}^n\subset\G_{k,d}$ is called a \emph{$t$-design} if the associated normalized atomic measure $\frac{1}{n}\sum_{l=1}^n \delta_{p_l}$ is a cubature measure of strength $t$ (see Definition \ref{tdes}), see \cite{Seymour:1984bh} for general existence results. Any sequence of $t_i$-designs $\{p^i_l\}_{l=1}^{n_i}\subset\G_{k,d}$ with $t_i\rightarrow\infty$ satisfies
\begin{equation}\label{eq:t is opt}
\varrho_i\asymp t_i^{-1},
\end{equation}
and moreover, the bound $n_i\gtrsim t_i^{k(d-k)}$ holds, cf.~\cite{Harpe:2005fk,Breger:2016rc}. To relate $n_i$ with $\varrho_i$ via $t_i$, a sequence of $t_i$-designs $\{p^i_l\}_{l=1}^{n_i}\subset\G_{k,d}$ is called a \emph{low-cardinality design sequence} if $t_i\rightarrow\infty$ and 
\begin{equation}\label{eq:n t etc}
n_i\asymp  t_i^{k(d-k)}, \quad i=1,2,\ldots
\end{equation}
For their existence and numerical constructions, we refer to \cite{Etayo:2016qq} and  \cite{Breger:2016vn, Breger:2016rc}. According to \cite{Breger:2016rc}, see also \eqref{eq:t is opt} and \eqref{eq:n t etc}, any low-cardinality design sequence $\{p^{i}_l\}_{l=1}^{n_i}$ covers asymptotically optimal, i.e., 
\begin{equation*}
\varrho_i\asymp n_i^{-\frac{1}{k(d-k)}}.
\end{equation*}
Benefiting from the covering property, we will use low-cardinality design sequences as a representation of the overall space of orthogonal projectors $\G_{k,d}$.

\subsection{Linear least squares fit}\label{lsfit}
With the linear least squares fit we can directly gain information about the relation between $\mathcal{M}(p,x)$ and $\tvar(px)$ for a given data set $x=\{x_i\}_{i=1}^m\subset \R^{d }$ when $p$ varies.  
Given the two samples
\begin{equation}\label{eq:two quant}
\{\tvar(p_1x),\ldots,\tvar(p_nx)\},\quad \{\mathcal{M}(p_1,x),\ldots,\mathcal{M}(p_n,x)\},
\end{equation}
the linear least squares fitting provides the best fitting straight line, 
\begin{equation*}
\tvar(p_lx) \approx s \cdot \mathcal{M}(p_l,x) + \gamma,\quad l=1,\ldots,n,
\end{equation*}
where $s$ and $\gamma$ are determined by the sample variances and the sample covariance. If $\{p_l\}_{l=1}^n$ is a $2$-design, then the sample (co)variances coincide with the respective population (co)variances for $P\sim\lambda_{k,d}$, see Appendix \ref{tdes2} for further details. It follows that 
\begin{align}
s &= \frac{\Cov(\mathcal{M}(P,x),\tvar(Px))}{\Var(\mathcal{M}(P,x))} \text{\quad with }P\sim\lambda_{k,d}, \label{spop} \\
 \gamma & =\tfrac{k}{d}\tvar(x) - s. \label{gammapop}
\end{align}
The quantities $s$ and $\gamma$ can be directly computed, where $\tvar(x)$ is given by \eqref{eq:total variance} and the covariances are stated in Corollary \ref{covariances}. Note that \eqref{spop} and \eqref{gammapop} are now independent of the particular choice of $\{p_l\}_{l=1}^n$. 

The correlation between the two samples \eqref{eq:two quant} yields additional information about their relation. As before, if $\{p_l\}_{l=1}^n$ is a $2$-design, then the sample correlation coincides with the population correlation \eqref{eq:bound from below cov} for $P\sim\lambda_{k,d}$, cf.~Appendix \ref{tdes2}. High correlation for a specific data set $x$ suggests that random projections and PCA preserve competing properties, whose benefits need to be assessed for the specific subsequent task.

\section{Numerical experiments in pattern recognition}\label{sec:4}
We investigate the impact on classification accuracy when applying specific orthogonal projections to input data. The real world data chosen yields a straightforward classification task, serving as a toy example for comparing the accuracy of several projected input data in simple learning frameworks. Projectors are chosen from a $t$-design in view of $\tvar(px)$ and $\mathcal{M}(p,x)$. For all computations made in this Section the `Neural Network' and `Statistics and Machine Learning' toolboxes in MatlabR2017a are used. 

We use the publicly available \texttt{iris} data set from the UCI Repository of Machine Learning Database suitable for supervised classification learning. It consists of $3$ classes with $50$ instances each, where each class refers to a type of iris plant. The instances are described by $4$ features resulting in the input samples $\{x_i\}_{i=1}^{150}\subset\R^4$ and target samples $\{y_i\}_{i=1}^{150}\subset \{0,1\}^3$. For comparison we classify the diverse input data with support vector machine (SVM) and $3$-layer neural networks (NN) with $5$ and $10$ hidden units (HU). 

\subsection{Choice of orthogonal projection}
In the experiment we use projections $p \in \G_{2,4}$ reducing the original dimension from $d=4$ to $k=2$. As a finite representation of the overall space, we use a t-design of strength $14$ from a low-cardinality sequence (see Section \ref{optcov}) consisting of $8475$ orthogonal projectors. Note that the dimension reduction in practice takes place by applying $q\in \mathcal{V}_{k,d}$ with $q^\top q = p\in\G_{k,d}$, where 
\begin{equation*}
\mathcal{V}_{k,d}:=\{ q\in\R^{k\times d} : qq^\top = I_k\}
\end{equation*}
denotes the Stiefel manifold. When taking norms, $p$ and $q$ are interchangeable, i.e., $\|q(x)\|^2 = \|p(x)\|^2$, for all $x\in\R^d$. Therefore we can use w.l.o.g. the theory developed for $p$. 

The projections are chosen in a deterministic manner viewing the previously described competing properties. In Figure \ref{resultsplot} the three quantities $\tvar(px)$, $\mathcal{M}(p,x)$ and $\mathcal{V}(p,x)$ are pairwise plotted for all projectors in $\{p_l\}_{l=1}^{8475}$. For comparison we choose the following projections $p \in \{p_l\}_{l=1}^{8475} \subset \G_{2,4}$, see Figure 4.\ref{figa} for a visualization.
\begin{itemize}
\item[$p_{\times}$] \ \ closest to the expected values $1$ and $\tfrac{k}{d} \tvar(x)$ (see \eqref{exptvar} and \eqref{expM}), 
\item[$p_{\Diamond}$] \ \ preserving $\mathcal{M}(p,x) \approx 1$ and maximizing $\tvar(px)$,
\item[$p_\text{\scalebox{0.8}{$\square$}}$] \ \ preserving $\mathcal{M}(p,x) \approx 1$ and minimizing $\tvar(px)$,
\item[$p_\text{\scalebox{0.6}{$\bigcirc$}}$] \ \ $\tvar(px) \approx \tvar(p_{\Diamond} x)$ and maximizing $\mathcal{M}(p,x)$,
\item[$p_{\Huge \star}$] \ \ minimal $\tvar(px)$,
\item[$p_{*}$] \ \ maximal $\tvar(px)$ (PCA).
\end{itemize}

\subsection{Results}

\begin{figure}
\subfigure[$\{\mathcal{M}(p_l,x),\tvar(p_l x)\}_{l = 1}^{8475}$]{
\includegraphics[width=.47\textwidth]{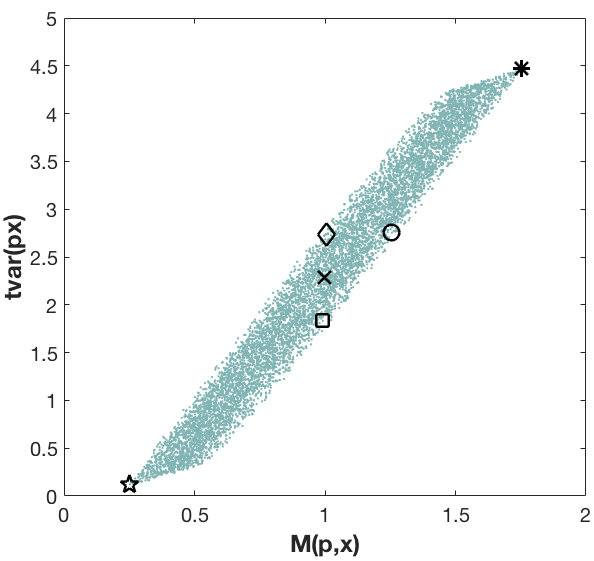} \label{figa}}
\subfigure[least squares fit]{
\includegraphics[width=.47\textwidth]{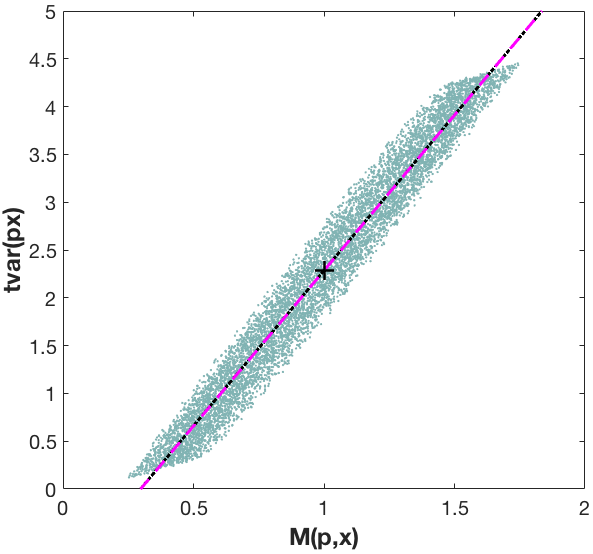} \label{figb} } 
\caption{Projections $\{p_l\}_{l=1}^{8475} \subset G_{2,4}$ from a t-design of strength $14$ evaluated on the iris data set $x \subset \R^{4 \times 150}$.}
\label{resultsplot}
\end{figure}

In Figure 4.\ref{figb} we see the linear least squares fitting line, computed directly and via the slope and intercept as stated in \eqref{spop} and $\eqref{gammapop}$. The correlation coefficient \eqref{eq:bound from below cov} is $0.98$, which suggests that preserving the two properties is highly competing and needs to be balanced. 

In Table \ref{tableres} the classification results of the iris data are presented. We can see that in this comparison the projector {$p_{\Diamond}$}, which corresponds to preserving $\mathcal{M}(p,x) \approx 1$ and maximizing $\tvar(px)$, yields the highest and most robust results. It even yields better results than working with the original input data. The projections that preserve $\mathcal{M}(p,x) \approx 1$ but do not take care of the magnitude of the total variance yield much worse results. On the other hand, the projections that just focus on high total variance still do not yield as high results as the projection {$p_{\Diamond}$} that balances both properties. 

\begin{table}
	\begin{tabular}{l|c|c|c}
	Input/Method & \textbf {NN (10 HU)} & \textbf{ NN (5 HU) }&  \textbf{SVM} \\
	\hline
	\textbf{ x }& (97.6, 1.25) & (97.5, 2.29)  & (96.7, 0.15) \\
	\hline
	\rowcolor{lightgray}\textbf{ p$_{\Diamond}$x }& {\color{red} (98.4, 0.42)}& {\color{red} (98.3, 2.15) }& {\color{red} (97.3, 0.06)}\\
	\hline
	\textbf{ p$_{\times}$x }& (88, 1.73)& (87.9, 1.92)  & (87.6, 0.63) \\
	\hline
	\textbf{ p$_\text{\scalebox{0.8}{$\square$}}$x }& (87.3, 9.56) & (86.9, 10.81) & (87.7, 0.42)  \\
	\hline
	\textbf{ p$_\text{\scalebox{0.6}{$\bigcirc$}}$x }& (96.8, 2.74) & (96.7, 1.77) & (96, 0.17) \\
	\hline 
	 \textbf{ p$_{*}$x }(PCA) & (96.9, 1.36) & (96.5, 4.07) & (96, 0.37) \\ 
	 \hline 
	  \textbf{ p$_{\Huge{\star}}$x }& (62.1, 44.30) & (58.9, 70.78) & (56, 0.61)  \\ 
	\end{tabular}
	\vspace{1ex}
\caption{Classification results of iris data, when using projected input data in support vector machine (SVM) and shallow neural networks (NN). Mean and Variance ($\times 10^{-4}$) of $1000$ independent NN runs, $100$ independent runs with $10$-fold cross-validation in SVM.}
\label{tableres}
\end{table}

\begin{remark}\label{rem:zweiter}
Given a data set $x$, the projector $p_{\Diamond}$ is a good choice to balance both objectives O1) and O2). It can be computed by directly analyzing $\{\tvar(p_1x),\ldots,\tvar(p_nx)\}$ and $\{\mathcal{M}(p_1,x),\ldots,\mathcal{M}(p_n,x)\}$ of a finite covering $\{p_l\}_{l=1}^n$ of $\G_{k,d}$. For higher dimensions an accurate representation of $\G_{k,d}$, in order to heuristically select $p_{\Diamond}$, requires large computational costs. The least squares regression line for a $2$-design, as stated in \ref{gammapop}, can be directly computed with low computational cost. This offers helpful information about the interplay between O1) and O2). 
\end{remark}

\section{Augmented target loss functions}\label{general}
In the previous section projectors were applied to input features of shallow neural networks. In more complex architectures, such as deep neural networks, the adaption of weights can be viewed as optimization of input features, e.g. arising features can be used for transfer learning \cite{Yosinski:2014:TFD:2969033.2969197}. Whereas the input data is processed and optimized in each iteration, the target data stays usually unchanged during the whole learning process, serving as measure of accuracy. The representation of the target data is one key property for successful approximation with neural networks. Here, we will introduce a general class of loss functions, i.e. augmented target (AT) loss functions, that use projections and features to yield beneficial representations of the target space, emphasizing important characteristics. 

In optimization problems additional penalty terms are used for regularization or to enforce other beneficial constraints. In deep learning, weight decay (i.e. Tikhonov regularization) is a standard adaption of the loss function to that effect. Incorporating additional underlying information via features of the output/target data has been studied in diverse settings tailored to particular imaging applications. Perceptual loss functions have been used in \cite{Ledig2017PhotoRealisticSI} for image super-resolution, incorporating the comparison of high-level image features that arise from pretrained convolutional neural networks, i.e. the VGG-network \cite{vgg}. Deep perceptual similarity metrics have been proposed in \cite{Dosovitskiy:2016:GIP:3157096.3157170} for generating images, comparing image features instead of the original images. In \cite{Johnson2016PerceptualLF} a similar approach was successfully used for style transfer and super-resolution, adding a network that defines loss functions. Anatomically constrained neural networks (ACNN) have been introduced in \cite{acnn} and applied to cardiac image enhancement and segmentation. Their loss functions incorporate structural information by using autoencoders to gain features about lower dimensional parametrization of the segmentation. Brain segmentation was studied in \cite{semantic}, where information about the desired structure has been added in the loss function via an adjacency matrix. It was used for fine-tuning the supervised learned network with unlabeled data, reducing the number of abnormalities in the segmentation. 

The information of certain target characteristics can be very powerful and even replace the need of annotations in some tasks. In \cite{Stewart2017LabelFreeSO} label-free learning is approached by using just structural information of the desired output in the loss function instead of annotated target values. 

In the following, we will define a general framework of loss functions that add information of target characteristics via features and projections in supervised learning tasks.

%
%
\subsection{General framework}

Let the training data be input vectors $\{x_i\}_{i=1}^m\subset \R^r$ with associated target values $\{y_i\}_{i=1}^m\subset \R^s$. We consider training a neural network 
\begin{equation*}
f_\theta:\R^r\rightarrow \R^{s},
\end{equation*}
where $\theta \in \R^N$ corresponds to the vector of all free parameters of a fixed architecture. In each optimization step for $\theta$, the network's output $\{\hat{y}_i=f_\theta(x_i)\}_{i=1}^m\subset\R^s$ is compared with the targets $\{y_i\}_{i=1}^m$ via an underlying loss function $L$. 

In contrast to ordinary learning problems with highly accurate target data, complicated learning tasks arising in many real world problems do not yield sufficient results when optimizing neural networks with standard loss functions $L$, such as the widely used mean least squares error 
\begin{equation}\label{eq:lsq}
L_{\MSE}(\{y_i\}_{i=1}^m,\{\hat{y}_i\}_{i=1}^m) := \frac{1}{m}\sum_{i=1}^m \norm {y_i - \hat{y}_i}^2.
\end{equation}
The training data may include important information that is obvious for humans, but poorly represented within the original target data and therefore lacks consideration in the learning process. To overcome this issue, we propose to add information tailored to the particular learning problem represented by additional features of the outputs and targets. 


First, we select transformations
\begin{equation*}
T_j:\R^s\rightarrow \R^t, \quad j=1,\ldots,d,
\end{equation*}
to enable error estimation in transformed output/target spaces. Note that the transformations $T_j$ are not required to be linear. However, they should be piecewise differentiable to enable subsequent optimization of the loss function with gradient methods. We shall allow for additional weighting of the transformations $T_1, \dots, T_d$ to facilitate the selection of features for a specific learning problem. The previous sections suggest that orthogonal projections can provide favorable feature combinations, which essentially turns into a weighting procedure.

To enable suitable projections, we stack the $d$ output/target features 
\begin{equation*}
T(y_i):=\begin{pmatrix}
    T_1 (y_i)^\top \\
    \vdots \\
    T_d (y_i)^\top
  \end{pmatrix}\in\R^{d\times t},
\end{equation*}
so that applying a projector $p\in\G_{k,d}$ to each column of $T(y_i)$ yields $p(T(y_i))\in\R^{d\times t}$. We now define the augmented target loss function with projections by
\begin{equation}\label{featloss}
L_{p}\big(\{y_i\},\{\hat{y}_i\}\big) := L(\{y_i\},\{\hat{y}_i\}) + \alpha \cdot \tilde{L}\big(\{p(T(y_i))\},\{p(T(\hat{y}_i))\}\big),
\end{equation}
where $\alpha > 0$ and $L$, $\tilde{L}$ correspond to conventional loss functions. Apparently, $L_p$ depends on the choice of $p\in\G_{k,d}$. The projection $p(T(y_i))$ weighs the previously chosen feature transformations $T(y_i)$. Standard choices of $L$ and $\tilde{L}$ are $L_{\MSE}$, in which case $L_p$ becomes
\begin{equation}\label{varlossfrob}
L_{p}\big(\{y_i\},\{\hat{y}_i\}\big)= \frac{1}{m}\sum_{i=1}^m \norm {y_i - \hat{y}_i}^2 + \alpha \cdot \frac{1}{m} \sum_{i=1}^m \| p(T(y_i)) - p(T (\hat{y}_i)) \|_{\F}^2.
\end{equation}

\begin{remark}
For $k=d$ the projector $p$ is the identity. In this case the transformations can map into different spaces, i.e. 
\begin{equation*}
T_j:\R^s\rightarrow \R^{t_j}, \quad j=1,\ldots,d,
\end{equation*}
and we can now write the standard augmented target loss function by 
\begin{equation}\label{AT}
L_{AT}\big(\{y_i\},\{\hat{y}_i\}\big) = \sum_{j=1}^d \alpha_j \cdot L^j \big(\{T_j(y_i)\},\{T_j(\hat{y}_i)\}\big),
\end{equation}
where $T_1$ corresponds to the identity function, $L^1,\ldots, L^d$ are common loss functions and $\alpha_1,\ldots, \alpha_d > 0$ are weighting parameters. 
\end{remark}

It should be mentioned that $\alpha$ resembles a regularization parameter. The actual minimization of \eqref{eq:lsq} among $\theta$ is usually performed through Tikhonov type regularization in many standard deep neural network implementations. The formulation \eqref{featloss} adds one further variational step for beneficial output data representation. 

\begin{remark}
Our proposed structure with target feature maps $T_1,\ldots,T_d$ as in \eqref{AT} relates to multi-task learning, which has been successfully used in deep neural networks \cite{Caruana1997}. It handles multiple learning problems with different outputs at the same time. In contrast to multi-task learning, we aim to solve a single problem but also penalize the error in transformed spaces enhancing certain target characteristics.
\end{remark}

For the projected feature transformations in the augmented target loss function it is not possible to identify a balancing projection $p$ heuristically (such as $p_\Diamond$ in Section \ref{sec:4}), because the output $y$ changes in each iteration when the loss function is called. In the following clinical numerical experiment we overcome this issue by using random projections and PCA in each optimization step and compare it to prior deterministic choices of projections. 

\section{Application to clinical image data}\label{sec:appl retina} 
The first experiment is a clinical problem in retinal image analysis of the human eye, where the disruptions of the so-called photoreceptor layers need to be quantified in optical coherence tomography images (OCT). The photoreceptors have been identified as the most important retinal biomarker for prediction of vision from OCT in various clinical publications, see e.g. \cite{octbio}. As OCT technology advances, clinicians are not able to look at each slice of OCT themselves (in mean they get 250 slices per patient and have 3-5 minutes/patients including their clinical examination). Therefore, automated classification of e.g. photoreceptor status is necessary for clinical guidance.

\subsection{Data and objective}
In this application, OCT images of different retinal diseases (diabetic macular edema and retinal vein occlusion) were provided by the Vienna Reading Center recorded with the Spectralis OCT device (Heidelberg Engineering, Heidelberg, Germany). Each patient's OCT volume consists of $49$ cross-sections/slices ($496 \times 512$ pixels) recorded in an area of $6 \times 6$ mm in the center of the human retina, which is the part of the retina responsible for vision. Each of the slices was manually annotated by a trained grader of the reading center. This is a challenging and time-consuming procedure that is not feasible in clinical routine but only in a research setting. The binary pixelwise annotations serve as target values, enabling a supervised learning framework. 

The objective is to accurately detect the photoreceptor layers and their disruptions pixelwise in each OCT slice by training a deep convolutional neural network with a suitable loss function. The learning problem is complicated by potentially inaccurate target annotations, as studies have shown that inconsistencies between trained graders are common, cf.~\cite{Shahrian}. Moreover, the learning task is unbalanced in the sense that there are many more slices showing none or very little disruptions. We shall observe that optimization with respect to standard loss functions performs poorly in regards to detecting disruptions. The augmented target loss function proposed in the previous section can enhance the detection.

\begin{figure}     
\subfigure[Healthy photoreceptor region]{\includegraphics[width = 0.49 \textwidth]{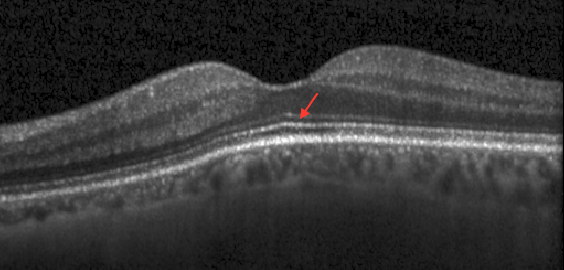}}
\subfigure[OCT slice plus manual annotation]{\includegraphics[width = 0.49 \textwidth]{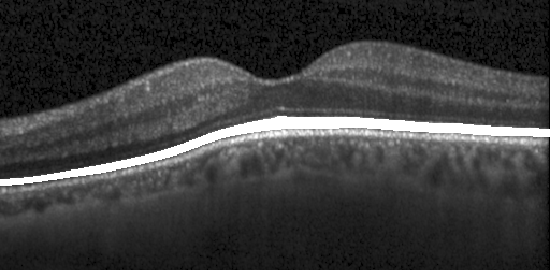}}
\subfigure[Disrupted photoreceptor region]{\includegraphics[width = 0.49 \textwidth]{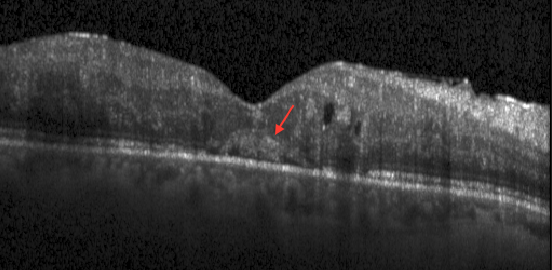}}
\subfigure[OCT slice plus manual annotation] {\includegraphics[width = 0.49 \textwidth]{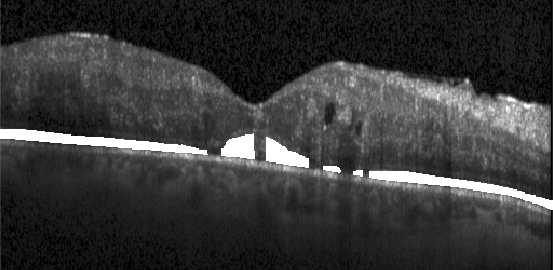}} 
\caption{OCT provides cross-sectional visualization of the human retina. }
\label{ex}
\end{figure}

\subsection{Convolutional neural network learning}
We implemented our experiments using Python 3.6 with Pytorch 1.0.0. A deep convolutional neural network $f_\theta$ is trained by applying the U-Net architecture reported in \cite{unet} with a sigmoid activation function and Tikhonov regularization. A set of $20$ OCT volumes ($980$ slices) from different patients with corresponding annotations are used for training, where $4$ volumes were used for calibration (validation set). Another $2$ independent test volumes were identified for evaluating the results, one without any disruptions in the photoreceptor layers, whereas the other one includes a high number of disruptions.

Each OCT slice is represented by a vector $x_i\in\R^{r}$ with $r=496\cdot 512$. The collection $\{x_i\}_{i=1}^m$ corresponds to all slices from the training volumes, i.e. $m=20\cdot 49$. Further matching the notation of the previous section, we have $r=s$ and $f_\theta:\R^r\rightarrow \R^r$ with binary target vectors $y_i\in \{0,1\}^r$. We observe that disruptions are not identified reliably when using the least squared loss function \eqref{eq:lsq}. To overcome this issues, we use the proposed augmented target loss function with least squared losses as stated in \eqref{varlossfrob}. 

To enhance disruptions within the output/target space, we heuristically choose $d=4$ local features of the original representation. They are derived from convolutions with $2$ edge filters, $T_1$ (Prewitt) and $T_2$ (Laplacian of Gaussian), and from $2$ Gaussian highpass filters, yielding $T_3$ and $T_4$. Note that these feature transformations keep the same size, i.e. $T_j:\R^r\rightarrow\R^r$ for $j=1,\ldots,d$. See Figure \ref{features} for example images. 

\begin{figure} \label{features}

\subfigure[Prewitt]{
      \includegraphics[width = 0.45\textwidth]{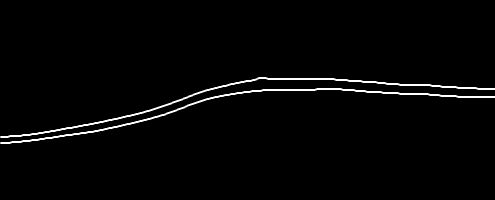}      
      \includegraphics[width = 0.45\textwidth]{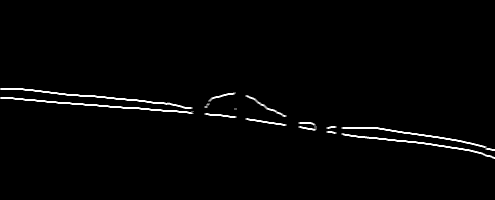} 
      }	
      
\subfigure[Laplacian of Gaussian (LoG)]{
       \includegraphics[width = 0.45\textwidth]{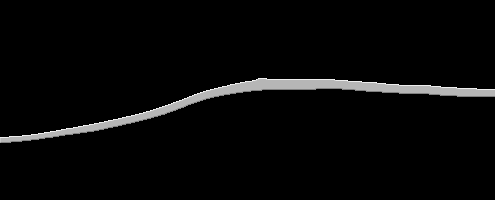} 
       \includegraphics[width = 0.45\textwidth]{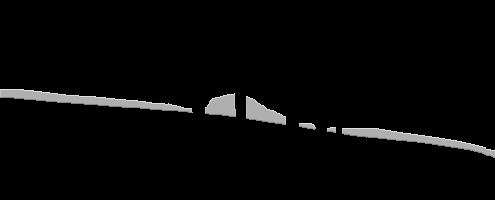} 
}

           \subfigure[Gaussian highpass (threshold = 40)]{
 		     \includegraphics[width = 0.45\textwidth]{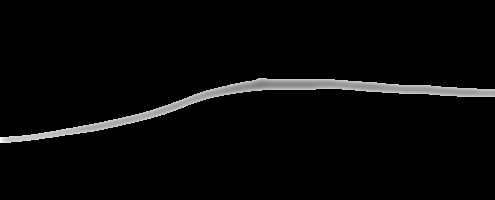} 
		       \includegraphics[width = 0.45\textwidth]{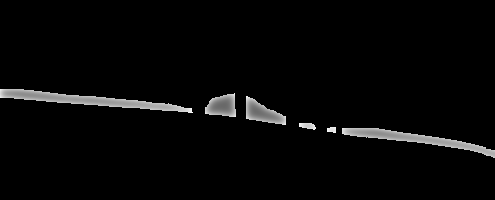} 
		       }	       
	
		       \subfigure[Gaussian highpass (threshold = 100)]{
 	\includegraphics[width = 0.45\textwidth]{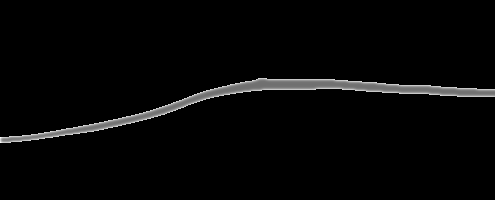}
	\includegraphics[width = 0.45\textwidth]{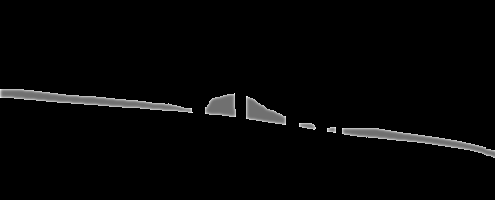} 
	}
\caption{Features on output and targets that enhance edges in different ways. It is not obvious which transformations are of most importance, weighting by projections can overcome this issue.}\label{fig:T}
\end{figure} 

We can derive several augmented target loss functions $L_p$ by choosing different $p \in\G_{k,d}$ for \eqref{featloss}. In this experiment we use the following projections: 
\begin{itemize}
\item $p=I_4$,
\item $\{p_l\}_{l=1}^{15}$, all projections from a t-design of strength $2$ $\subset \G_{2,4}$ (see \cite{Breger:2016vn}),
\item $p_{\PCA} \in \G_{2,4}$, projection determined by PCA on the training data,
\item $p_{\lambda_{2,4}}$, random projection chosen according to $\lambda_{2,4}$ in each mini-batch. 
\end{itemize}

\subsection{Results}
Since the detection problem is highly unbalanced we use precision/recall curves~\cite{davis2006relationship} for evaluating the overall performance of each loss function model. The area under the curve (AUC) was used as a numerical indicator of the success rate, \cite{Pabst:1999kr}. The higher the AUC the better the classification.  

The results of the different loss functions on the independent test set are stated in Table \ref{table}. Due to the imbalance within the data, the photoreceptor region is identified well, but disruptions are not identified reliably when using the least squared loss function \eqref{eq:lsq}. For $\alpha = 0.1$ all proposed augmented target loss functions $L_p$ clearly increase the success rate of the disruption quantification. Note that all projections are independent from the actual data set, except PCA that was computed beforehand on the training data. 

The features itself (i.e. $p = I_4$) improve the quantification and weighting them by projections increases the results even more: using the fixed projection $p_{12}$ from the t-design sequence $\{p_l\}_{l=1}^{15}$ on the output/target features yields the highest accuracy for photoreceptors and disruptions. This corresponds to the results of the previous sections, stating that depending on the particular data there are projections in the overall space acting beneficially. 

Since this projection generally cannot be found beforehand, using random projections in each loss function evaluation step is easier possible in practice and independent from the data. They can be computed very efficiently and randomization can generalize and robusten the information, cf. \cite{MAL-035}. In the following we will view a second classification problem based on spectrograms, where augmented target loss functions with random projections can improve the accuracy.

\begin{table}
\caption{Comparison of AUC values for photoreceptors segmentation and disruption detection.}
\label{table}
\vspace{0.4cm}
\centering
\begin{tabular}{ l | c | c}
  \hline
  \textbf{Loss function} & \textbf{Photoreceptors}& \textbf{Disruptions} \\
  \hline	
   \hline  
  $L_{\MSE}$ &	0.9720 &	0.4399  \\		
  \hline
  \hline
  $L_p$ & \\
    \hline
  $p = I_4$ & 0.9736  & 0.4686 \\
  \hline
  $p_{\lambda_{2,4}}$   &  0.9746  & 0.4720\\
  \hline  
  $p_{\PCA}$    & 0.9716 &  0.5331 \\
  \hline  
  \rowcolor{lightgray} $p_{12} $  & {\color{red} \ 0.9755 }& {\color{red} 0.5558}\\
  \hline  
\end{tabular}
\end{table}

\section{Application to musical data}\label{sec:appl music} 
Here, the learning task is a prototypical problem in Music Information Retrieval, namely multi-class classification of musical instruments. In analogy to the MNIST problem in image recognition, this classification problem is commonly used as a basis of comparison for innovative methods, since the ground truth is unambiguous and sufficient annotated data are available. The input to the neural network are spectrograms of audio signals, which is the standard choice in audio machine learning. Spectrograms are calculated from the time signal using a short-time Fourier transform and taking the absolute value squared of the resulting spectra, thus yielding a vector for each time-step and a two-dimensional array, like an image, cf.~\cite{badogr17}. 

Reproducible code and more detailed information of our computational experiments can be found in the online repository \cite{:la}. 

\subsection{Data and objective}
The publicly available GoodSounds dataset \cite{GoodSounds} contains recordings of single notes and scales played by several single instruments. To gain equally balanced input classes we restrict the classification problem to $6$ instruments: clarinet, flute, trumpet, violin, alto saxophone and cello. Note that the recordings are monophonic, so that each recording yields one spectrogram that we aim to correctly assign to one of the $6$ instruments. 

After removing the silence \cite{sox, SoxSilenceTut}, segments from the raw audio files are transformed into log mel spectrograms \cite{librosa0.6.2}, so that we obtain images of time-frequency representations with size $100 \times 100$. One example spectrogram for each class of instruments is depicted in Figure \ref{fig:mels}. 

\begin{figure}     
\subfigure[Clarinet]{\includegraphics[width = 0.32 \textwidth]{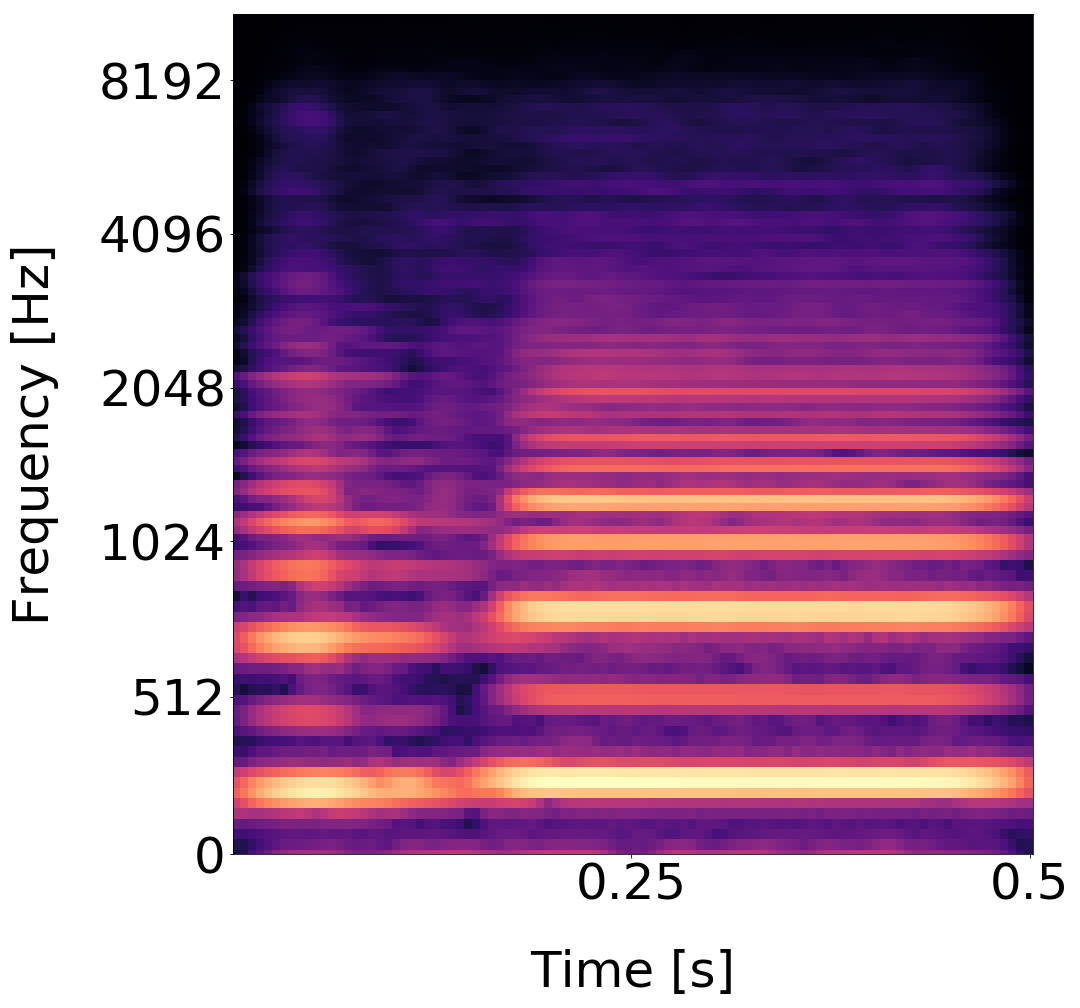}}
\subfigure[Flute] {\includegraphics[width = 0.32 \textwidth]{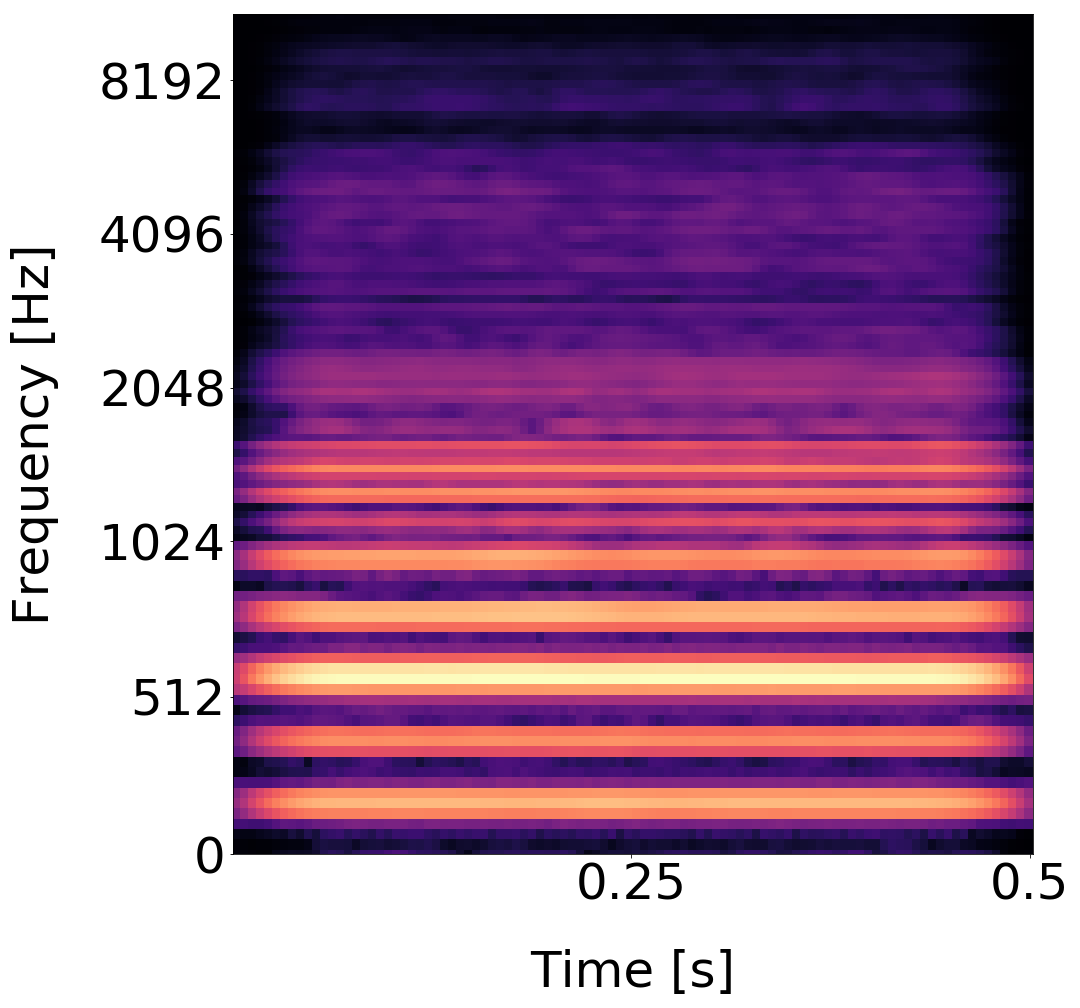}} 
\subfigure[Trumpet]{\includegraphics[width = 0.32 \textwidth]{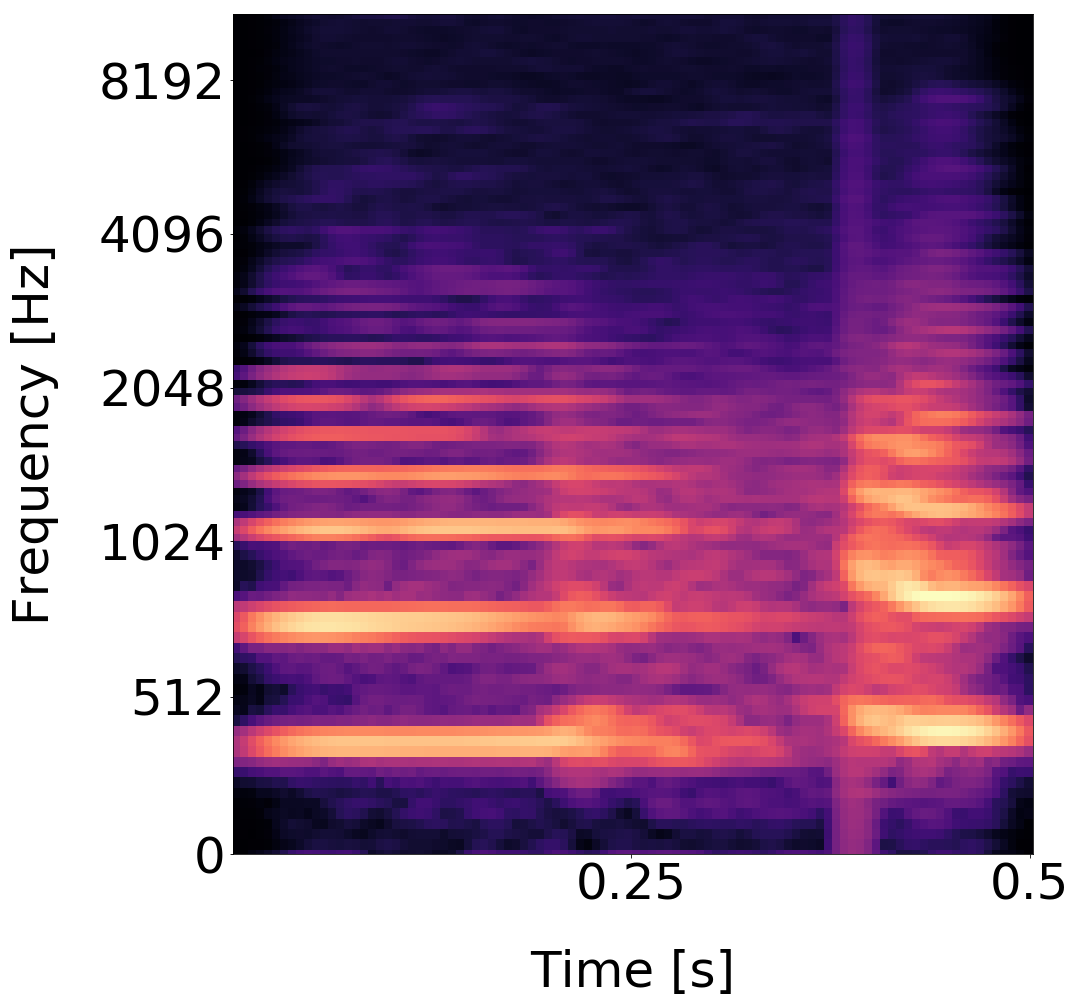}}

\vspace{2ex}
\subfigure[Violin]{\includegraphics[width = 0.32 \textwidth]{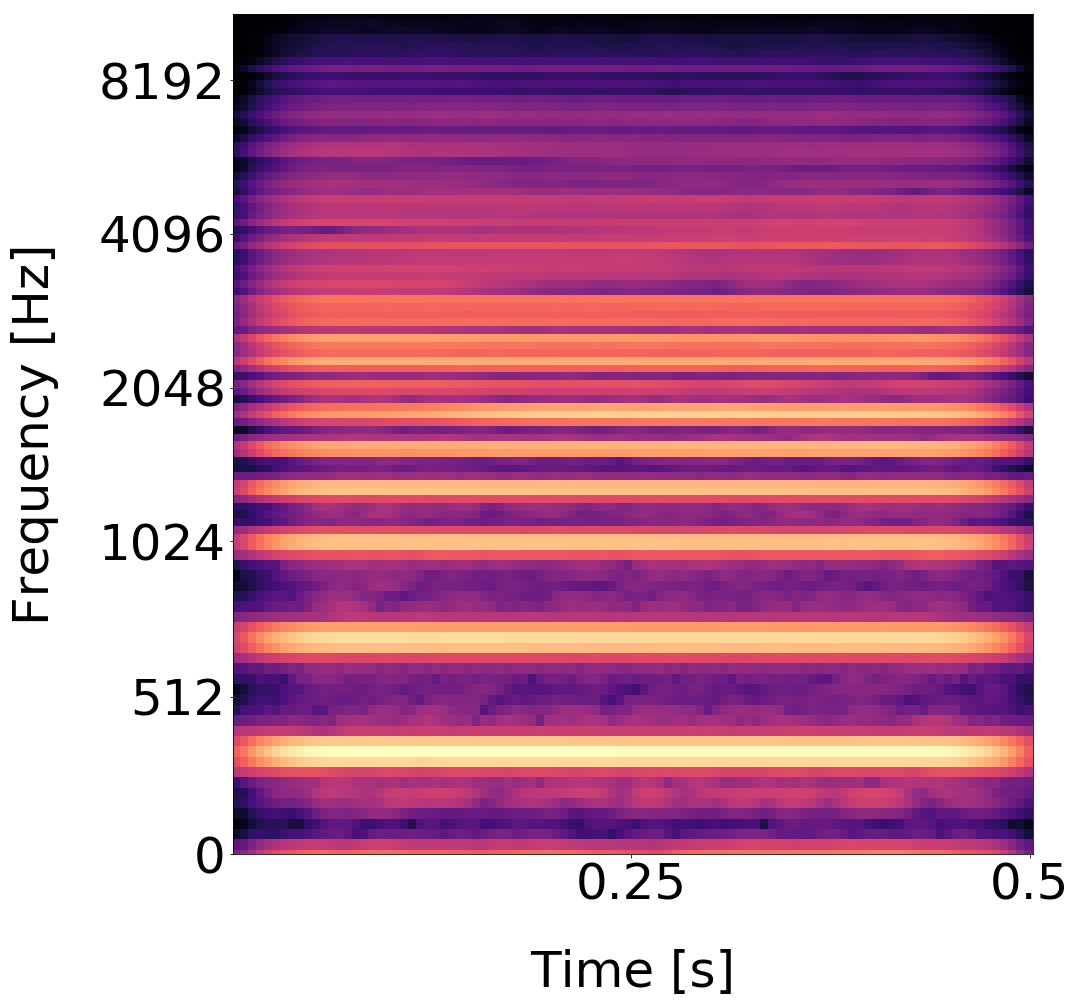}}
\subfigure[Alto Saxophone]{\includegraphics[width = 0.32 \textwidth]{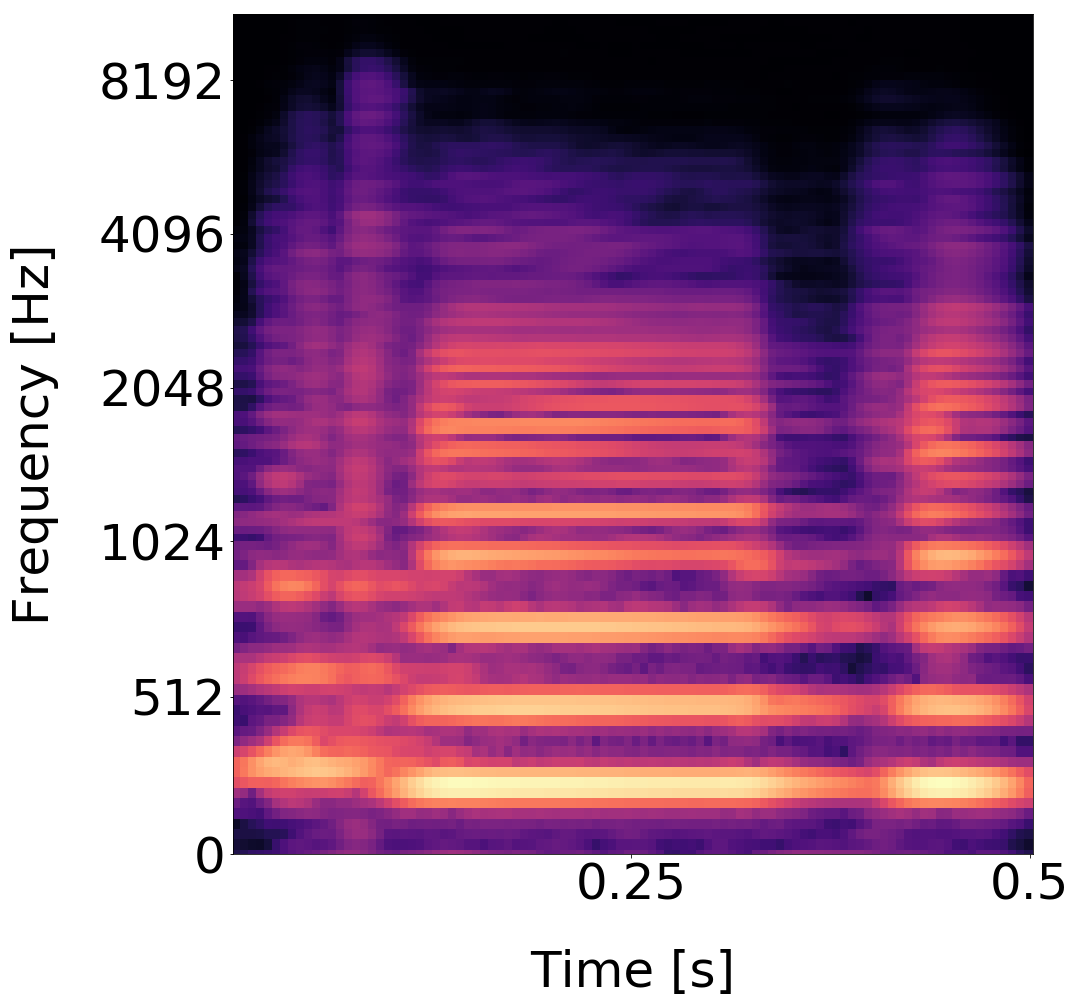}}
\subfigure[Cello]{\includegraphics[width = 0.32 \textwidth]{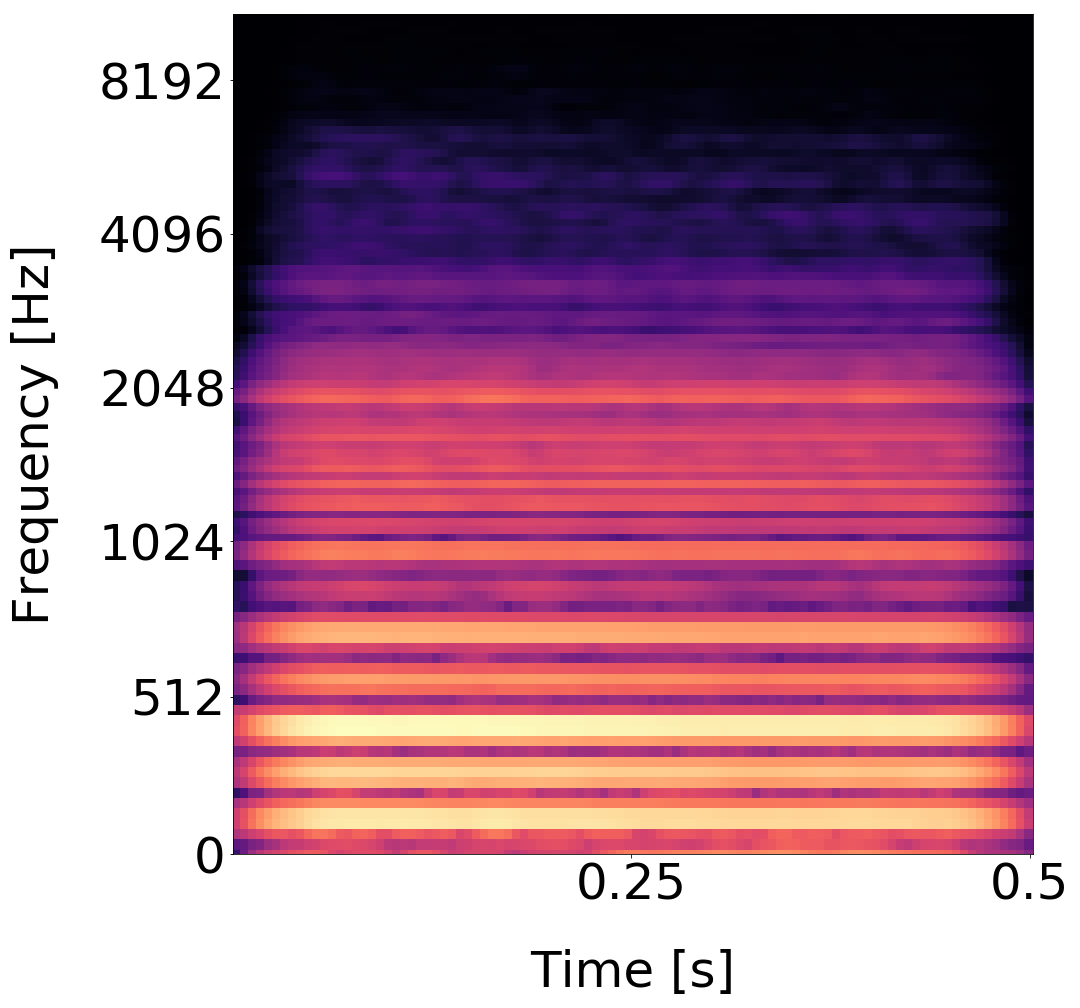}}
\caption{Log mel spectrograms of the $6$ different instruments. Intensities range from zero (black) to 1 (yellow).}
\label{fig:mels}
\end{figure}

\subsection{Convolutional neural network learning}
We implemented a fully convolutional neural network $f_{\theta}: \R^r \to [0,1]^s$, cf.~\cite{FullyConvolutional}, where $r = 100 \times 100$ and $s = 6$, in Python 3.6 using Keras 2.2.4 framework \cite{chollet2015keras} and trained it on the Nvidia GTX 1080 Ti GPU. The data is split into $140\,722$ training, $36 \,000$ validation and $36\,000$ independent test samples. We heuristically choose $d = 16$ output features arising directly from the particular output class. The transformations $T_1,\dots, T_{16}$, with $T_j: \R^6 \to \R$ for $j = 1, \dots, 16$, are then given by the inner product of the output/target and the feature vectors. Amongst others the features are chosen from the enhanced scheme of taxonomy \cite{von1961classification} and from the table of frequencies, harmonics and under tones \cite{FrequencyRanges}. We use the proposed augmented target loss function $L_p$ \eqref{featloss}, where $L_1$ corresponds to the categorical-cross-entropy loss \cite{Zhang:2018th} and $L_2$ to the mean squared error as in \eqref{varlossfrob}. We consider here two choices of $p$: the identity $I_{16}$ and random projectors $p\sim\lambda_{6,16}$ in $\G_{6,16}$. 

The deep learning model is sensitive to various hyper-parameters, including $\alpha$ and $p$, in addition to conventional parameters, such as the number of convolutional kernels, learning rate and the parameter $\beta$ for Tikhonov regularization. To find the best choices in a fair trial we utilize a~random hyper-parameter search approach, where we train $60$ models and select the $3$ best ones for a more precise search over different $\alpha$ in the augmented target loss function and $\beta$ for Tikhonov regularization. This results in $212$ models that are evaluated on the training and validation set. 
Finally, we select the best model based on the accuracy of the validation set and evaluate it on the independent test set. For comparison we also evaluate this model with no Tikhonov regularization, i.e. $\beta = 0$, see Table \ref{table:1}. 

\begin{table}
\begin{center}
\begin{tabular}{| r | r | c | r | r | r |}
\hline
$\alpha$ & $\beta$ & $p$ & training & test data\\
\hline
\hline
0	    & 0	    &    -  & 0.5541	& 0.5716\\
0.01	& 0	    &   $I_{16}$    & 0.5650		& 0.5683\\
\rowcolor{lightgray} 0.01	& 	0    &  $p_{\lambda_{6,16}}$   & 0.7722	& {\color{red}  0.7657}\\
0	& 0.05    &  -     & 0.9771		& 0.9729\\
0.01	& 0.05	&  $I_{16}$ & 0.9849	& 0.9802\\
\rowcolor{lightgray}0.01	&  0.05	& $p_{\lambda_{6,16}}$   & 0.9857   &  {\color{red} 0.9833}\\
\hline
\end{tabular}
\end{center}
\vspace{0.5em}	
\caption{Classification results with different parameter choices. The standard inbuilt Tikhonov regularization ($\ell_2$-norm of $\theta$) is weighted by $\beta$. For $\alpha>0$ the feature transformations $\{T_j\}_{j=1}^{16}$ are used in the loss function, either directly or weighted by a random projection $p_{\lambda_{6,16}}$. The accuracy of the model is measured by the number of correctly classified samples divided by the number of all samples.}\label{table:1}
\end{table}

\subsection{Results}
Table \ref{table:1} shows that no regularization and no features provide the poorest results. It seems that adding features with random projections have a regularizing effect and improve the results significantly. As expected, it is important to include Tikhonov regularization on $\theta$. Further enhancement happens by adding features via the modified augmented target loss function with or without additional weighting from projections. All results are very stable and are generalizing very well from training to the independent test set, see \cite{:la} for further details.

\section{Acknowledgment}
This work was partially funded by the Vienna Science and Technology Fund (WWTF) through project VRG12-009, by WWTF AugUniWien/FA746A0249, by International Mobility of Researchers (CZ.02.2.69/0.0/0.0/16 027/0008371), and by project LO1401. For the research, infrastructure of the SIX Center was used.

\newpage
\appendix

\section{Proof of Theorem \ref{theorem}}\label{app:2}
\subsection{Proof of \eqref{thm1} in Theorem \ref{theorem}}
For $\{y_i\}_{i=1}^M\subset\R^d$ and $p\in\mathcal{G}_{k,d}$, we define
\begin{equation}\label{eq:f for all}
f(p,\{y_i\}_{i=1}^{M}) :=  \tfrac{1}{M} \sum_{i=1}^M \tfrac{d}{k}  \|p(y_i)\|^2. 
\end{equation}
Given two sets, $\{y_i\}_{i=1}^{M_1}, \{z_j\}_{j=1}^{M_2}\subset \mathbb{R}^d$, suppose that $P\in\mathcal{G}_{k,d}$ is a random matrix, distributed according to a cubature measure of strength at least $2$. The covariance is given by
\begin{align*}
\Cov(f(P,\{y_i\}_{i=1}^{M_1}) , f(P,\{z_j\}_{j=1}^{M_2})) =& \\
&\hspace{-3.5cm}  \EE[(f(P,\{y_i\}) - \EE[f(P,\{y_i\})])(f(P,\{z_i\}) - \EE[f(P,\{z_i\})]
\end{align*}
Using the identity, cf. ~\cite{Bachoc:2010aa}, 
\begin{align}
 \tfrac{d}{k}  \EE \big[  \norm{Py}^2 \big] &= \norm{y}^2 ,\label{exp1}
\end{align}
directly yields
\begin{align*}
\Cov(f(P,\{y_i\}_{i=1}^{M_1}) , f(P,\{z_j\}_{j=1}^{M_2})) =& \\ 
&\hspace{-3.5cm} \EE[(\tfrac{1}{M_1} \sum_{i=1}^{M_1} \tfrac{d}{k}  \|P(y_i)\|^2 - \tfrac{1}{M_1} \sum_{i=1}^{M_1} \|y_i\|^2)(\tfrac{1}{M_2} \sum_{i=1}^{M_2} \tfrac{d}{k}  \|P(z_i)\|^2 - \tfrac{1}{M_2} \sum_{i=1}^{M_2} \|z_i\|^2)]
\end{align*}
Following \cite[Theorem 2.4, Section 3.1]{Graf:2014qd} we use that 
\begin{align}
\EE \big[ \norm{Py}^2 \norm{Pz}^2]& = \frac{1}{q} ( \alpha_1 \norm{y}^2 \norm{z}^2 + \alpha_2 \langle y,z \rangle^2), \quad y,z\in\R^d,\label{exp2} 
\end{align}
holds, where $q = (d-1)d(d+2)$, $\alpha_1 = (d+1)k^2 -2k$ and $\alpha_2 = 2k(d-k)$. This leads to the explicit formula of the population covariance
\begin{equation}\label{eq:cov universal}
\begin{split}
\Cov(f(P,\{y_i\}_{i=1}^{M_1}) , f(P,\{z_j\}_{j=1}^{M_2})) =& \\
 &\hspace{-3.5cm} \frac{a_{k,d}}{M_1M_2}\sum_{i=1}^{M_1}\sum_{j=1}^{M_2} \langle y_i,z_j\rangle^2-\frac{a_{k,d}}{d} \Big(\frac{1}{M_1} \sum_{i=1}^{M_1} \norm{y_i}^2\Big)\Big(\frac{1}{M_2} \sum_{j=1}^{M_2} \norm{z_j}^2\Big),
\end{split}
\end{equation}
with $a_{k,d}=  \tfrac{2d(d-k)}{k(d-1)(d+2)}$. \smallskip \\ 
\noindent
For $y:=\{y_i\}_{i=1}^M\subset\R^d\setminus\{0\}$ we set $\hat{y}_i:=\frac{y_i}{\|y_i\|}$, for $i=1,\ldots,M$. The identity \eqref{eq:cov universal} enables us to compute the population correlation 
\begin{equation}\label{eq:quotient for corr}
 \Corr(f(P,y),f(P,\hat{y})) = \frac{\Cov(f(P,y) , f(P,\hat{y}))}{\sqrt{\Var(f(P,y))} \sqrt{\Var(f(P,\hat{y}))}}
 \end{equation}
by the explicit formulas
\begin{align*}
\Cov\big[f(P,y) , f(P,\hat{y})\big] &= \frac{a_{k,d}}{M^2}\sum_{i,j=1}^M \langle y_i,\hat{y}_j\rangle^2-\frac{a_{k,d}}{d} \cdot \frac{1}{M} \sum_{i=1}^M \norm{y_i}^2 \\
 \Cov\big[f(P,y) , f(P,y)\big] &= \Var[f(P,y)]  = \frac{a_{k,d}}{M^2} \sum_{i,j=1}^M \langle y_i, y_j \rangle^2-\frac{a_{k,d}}{d}  \Big(\frac{1}{M} \sum_{i=1}^M \|y_i\|^2\Big)^2\\
\Cov\big[f(P,\hat{y}) , f(P,\hat{y})\big] &= \Var[f(P,\hat{y})] =  \frac{a_{k,d}}{M^2} \sum_{i,j=1}^M \langle \hat{y}_i, \hat{y}_j \rangle^2-\frac{a_{k,d}}{d}.
\end{align*}

Since the variance is always nonnegative and $\tfrac{a_{k,d}}{d} > 0$, the denumerator of $\Corr(f(P,y),f(P,\hat{y}))$ in \eqref{eq:quotient for corr} satisfies 
\begin{align*}
\sqrt{\Var(f(P,y))} \sqrt{\Var(f(P,\hat{y}))} & \leq \sqrt{\Big(\frac{a_{k,d}}{M^2} \sum_{i,j=1}^M \langle y_i, y_j \rangle^2 \Big) \Big( \frac{a_{k,d}}{M^2} \sum_{i,j=1}^M \langle \hat{y}_i, \hat{y}_j \rangle^2\Big)}\\
& \leq \frac{a_{k,d}}{M^2}  \sqrt{  \Big(\sum_{i,j=1}^M \langle y_i, y_j \rangle^2 \Big) \Big( \frac{1}{\min_i(\|y_i\|)^4} \sum_{i,j=1}^M \langle y_i, y_j \rangle^2\Big)}\\
& \leq \frac{1}{\min_i(\|y_i\|)^2 }  \frac{a_{k,d}}{M^2} \sum_{i,j=1}^M \langle y_i, y_j \rangle^2 .
\end{align*}
The enumerator of $\Corr(f(P,y),f(P,\hat{y}))$ in \eqref{eq:quotient for corr} is estimated by
\begin{align*}
\Cov(f(P,y) , f(P,\hat{y})) &\geq   \frac{a_{k,d}}{\max_i(\|y_i\|)^2}\frac{1}{M^2}\sum_{i,j=1}^M \langle y_i,y_j\rangle^2-\frac{a_{k,d}}{d} \max_i(\|y_i\|)^2 
\end{align*}
For $d\geq M$, a short calculation yields $\Cov(f(P,y) , f(P,\hat{y}))\geq 0$, so that we obtain
\begin{align*}
\Corr(f(P,y),f(P,\hat{y}))  &\geq  \frac{\min_i(\|y_i\|)^2}{\max_i(\|y_i\|)^2}-\frac{\min_i(\|y_i\|)^2 \max_i(\|y_i\|)^2}{ \frac{d}{M^2} \sum_{i,j=1}^M \langle y_i, y_j \rangle^2}. 
\intertext{The lower bound $\sum_{i,j=1}^M \langle y_i,y_j\rangle^2 \geq M\min_i(\|y_i\|)^4$ yields}
\Corr(f(P,y),f(P,\hat{y})) &\geq \frac{\min_i(\|y_i\|)^2}{\max_i(\|y_i\|)^2}- \frac{M}{d}\cdot\frac{\max_i(\|y_i\|)^2}{\min_i(\|y_i\|)^2}.
\end{align*}
Since the correlation is scaling invariant the choice $y=\{x_i-x_j:1\leq i<j\leq m\}$ with $M=\frac{m(m-1)}{2}$ implies \eqref{thm1} in Theorem \ref{theorem}. Incorporating the correct scaling yields the following corollary: 

\begin{corollary}\label{covariances}
For a given data set $x = \{x_i\}_{i=1}^{m}$ and for random $P\in\G_{k,d}$ the (co)variances of $\tvar(Px)$ \eqref{tvarpx} and $\mathcal{M}(P,x)$ \eqref{mpx} are given by
\small{ \begin{align*}
\Cov(\mathcal{M}(P,x),\tvar(Px)) &= \frac{k}{2d} \Big( \frac{a_{k,d}}{M^2}\sum_{i<j}\sum_{l<r} \big \langle x_i-x_j,\frac{x_l-x_r}{\|x_l-x_r\|} \big \rangle^2-\frac{a_{k,d}}{d} \cdot \frac{1}{M} \sum_{i<j} \norm{x_i-x_j}^2 \Big),\\
\Var(\tvar(Px)) &=  \frac{k^2}{4d^2} \Big(\frac{a_{k,d}}{M^2} \sum_{i<j} \sum_{l<r} \big \langle x_i-x_j, x_l-x_r \big \rangle^2-\frac{a_{k,d}}{d}  \big(\frac{1}{M} \sum_{i<j} \|x_i - x_j\|^2 \big)^2 \Big),\\
\Var(\mathcal{M}(P,x)) & = \frac{a_{k,d}}{M^2} \sum_{i<j}\sum_{l<r} \big \langle \frac{x_i-x_j}{\|x_i-x_j\|}, \frac{x_l-x_r}{\|x_l-x_r\|} \big \rangle^2-\frac{a_{k,d}}{d},
\end{align*} }
where $M=\frac{m(m-1)}{2}$ and $a_{k,d} =  \frac{2d(d-k)}{k(d-1)(d+2)}$. 
\end{corollary}

\subsection{Proof of the second part of Theorem \ref{theorem}}
For fixed parameters $\mu > 0,\sigma^2> 0$, that do not depend on $d$, let $Y_1\in\R^d$ be a random vector, whose squared entries are independent, identically distributed with mean $\EE Y_{1,l}^2=\mu$ and variance $\Var(Y_{1,l}^2)=\sigma^2$, for $l=1,\ldots,d$. We immediately observe
\begin{equation*}
\EE \Big(\frac{\|Y_{1}\|^2}{\sqrt{d}} \Big)= \sqrt{d}\mu,\qquad \Var\Big(\frac{\|Y_1\|^2}{\sqrt{d}}\Big) = \sigma^2.
\end{equation*}
For any $c>0$, Chebychev's inequality yields
\begin{equation*}
\mathbb{P}\Big(\Big|\frac{\|Y_{1}\|^2}{\sqrt{d}}-\sqrt{d}\mu  \Big|\geq c \sigma\Big)\leq \frac{1}{c^2}.
\end{equation*}
Suppose that $Y_2,\ldots,Y_M$ are copies of $Y_1$, not necessarily independent. Then the union bound
\begin{equation*}
\mathbb{P}\Big(\Big|\frac{\|Y_{i}\|^2}{\sqrt{d}}-\sqrt{d}\mu  \Big|\geq c \sigma, \;\text{for some $i=1,\ldots,M$}\Big)\leq \frac{M}{c^2}
\end{equation*}
implies that 
\begin{equation*}
\sqrt{d}\mu-c\sigma \leq \frac{\min_i(\|Y_i\|)^2}{\sqrt{d}}\leq \frac{\max_i(\|Y_i\|)^2}{\sqrt{d}}\leq \sqrt{d}\mu+c\sigma
\end{equation*}
holds with probability at least $1-\frac{M}{c^2}$. Provided that $\sqrt{d}\mu\neq c\sigma$ and $0 < \sqrt{d} \mu - c \sigma$, we deduce
\begin{equation*}
\frac{\sqrt{d}\mu-c\sigma}{\sqrt{d}\mu+c\sigma}\leq \frac{\min_i(\|Y_i\|)^2}{\max_i(\|Y_i\|)^2}\leq \frac{\sqrt{d}\mu+c\sigma}{\sqrt{d}\mu-c\sigma}.
\end{equation*}
We can choose $c = \tfrac{\mu}{\sigma} \sqrt[4]{d}$, since $0 < c \leq \tfrac{\sqrt[4]{d} \mu}{\sigma} \leq \tfrac{\sqrt{d} \mu}{\sigma}$. That directly yields
\begin{equation*}
\frac{1-\frac{1}{\sqrt[4]{d}}}{1+\frac{1}{\sqrt[4]{d}}}\leq\frac{\min_i(\|Y_i\|)^2}{\max_i(\|Y_i\|)^2}\leq\frac{1+\frac{1}{\sqrt[4]{d}}}{1-\frac{1}{\sqrt[4]{d}}}
\end{equation*}
holds with probability at least $1-\frac{\mu^2 M}{\sigma^2 \sqrt{d}}$. \\
It follows directly that $\frac{\min_i(\|Y_i\|)^2}{\max_i(\|Y_i\|)^2}$ converges towards $1$ in probability for $d\rightarrow\infty$,
\\
The choice $\{Y_1,\ldots,Y_M\}=\{X_i-X_j:1\leq i<j\leq m\}$ implies the second part of Theorem \ref{theorem}.

\subsection{Calculations for population covariances} \label{tdes2}
We notice that $\|p(x_i - x_j)\|^2=\trace(px_ix_i^\top - px_jx_j^\top)$ is a polynomial of degree $1$ in $p$. Hence, $\tvar(px)$ in \eqref{tvarpx} is also a polynomial of degree $1$ in $p$. If $\{p_l\}_{l=1}^n$ is a $1$-design, then the sample mean of $\{\tvar(p_1x),\ldots,\tvar(p_nx)\}$ satisfies
\begin{equation*}
\frac{1}{n}\sum_{l=1}^n \tvar(p_lx) = \mathbb{E}\tvar(Px),
\end{equation*}
which is the population mean of $\tvar(Px)$, with $P\sim\lambda_{k,d}$. Similarly, the term $\|p(x_i - x_j)\|^4$ is a polynomial of degree $2$ in $p$, so that $(\mathcal{M}(p,x))^2$ in \eqref{mpx} is a polynomial of degree $2$ in $p$. If $\{p_l\}_{l=1}^n$ is a $2$-design, then we derive
\begin{align*}
\sum_{l=1}^n (\mathcal{M}(p_l,x))^2- \Big(\sum_{j=1}^n\mathcal{M}(p_l,x)\Big)^2 = \mathbb{E}(\mathcal{M}(P,x))^2-\mathbb{E} \Big(\sum_{j=1}^n\mathcal{M}(P,x)\Big)^2,
\end{align*}
with $P\sim\lambda_{k,d}$. In other words, the sample variance of $\{\mathcal{M}(p_1,x),\ldots,\mathcal{M}(p_n,x)\}$ coincides with the population variance $\Var(\mathcal{M}(P,x))$. Analogously, we deduce that the sample covariance of \eqref{eq:two quant} coincides with the population covariance $\Cov(\mathcal{M}(P,x),\tvar(Px))$ with $P\sim \lambda_{k,d}$.

\bibliographystyle{amsplain}
\bibliography{biblio_ehler3}
\end{document}